\documentclass[12pt]{amsart}
\usepackage{amssymb,latexsym, amsthm,epsf}
\newif\ifXY\XYtrue
\ifXY
\input xy
\usepackage{xy}
\xyoption{all} %
\xyoption{curve} %
\xyoption{arc} \fi % For \ifXY

\usepackage{xy}
\usepackage{amsfonts}
\usepackage{amsmath, amsthm}
\usepackage{amssymb, amsxtra}

\usepackage{graphicx}

\usepackage{amscd}
\usepackage{multicol}
\usepackage{epstopdf}
\usepackage{color}
\usepackage{xparse}

\usepackage{comment}

\usepackage{epsf}
\usepackage{float}
\usepackage{extarrows}
\usepackage[a4paper, text={.8\paperwidth,.83\paperheight}, ratio=1:1]{geometry}
\usepackage[format=hang,font=small,textfont=it]{caption}

\usepackage[english]{babel}
\usepackage{latexsym}
\usepackage{nicefrac}

\usepackage{tikz}
\usetikzlibrary{calc}
\usepackage[tight]{subfigure}
\usepackage{afterpage}
\usepackage{hyperref}

\long\def\forget#1\forgotten{}
\usepackage{color}
\definecolor{greenbean}{RGB}{199,237,204}
\newtheorem{thm}{Theorem}[section]
\newtheorem{lemma}[thm]{Lemma}

\newtheorem{cor}[thm]{Corollary}

\newtheorem{notation}{Notation}[]
\newtheorem{Def}[thm]{Definition}

\newtheorem{construction}{Construction}

\newtheorem{rmk}[thm]{Remark}

\def\R{{\mathbb R}}

\def\C{{\mathbb C}}

\def\Z{{\mathbb Z}}
\def\P{{\mathbb P}}
\def\F{{\mathbb F}}

\def \ta{\tau}
\def \ta1{\tau_1}

\def \G{\gamma}

\newcommand\sg[1]{{\langle{#1}\rangle}}
\newcommand\nsg[1]{{\sg{\sg{#1}}}}
\newcommand\card[1]{{\left|#1\right|}}
\newcommand\co{{\,{:}\,}}

\def \CP{\mathbb{C} \mathbb{P}}

\def \CP{\mathbb{C} \mathbb{P}}

\newcommand\Galois[1]{{#1_{\operatorname{Gal}}}}

\def \prodl{\prod\limits}

\newcommand\Cref[1]{{Corollary~\ref{#1}}}

\newcommand\set[1]{{\{{#1}\}}}
\newcommand\Set[2]{{(\set{#1} \times \set{#2})}}

\newcommand{\ug}[1]{\G_{#1}}

\makeatletter
\newcommand{\uGammaSq}[2]{\begin{equation}\label{#1}
	\ug{#2}^2\uGammaSqChecknextarg}
\newcommand{\uGammaSqChecknextarg}{\@ifnextchar\bgroup{\uGammaSqGobblenextarg}{ = e,
		\end{equation} }}
\newcommand{\uGammaSqGobblenextarg}[1]{ = \ug{#1}^2\@ifnextchar\bgroup{\uGammaSqGobblenextarg}{  = e
\end{equation}}}
\makeatother
\newcommand\Dref[1]{{Definition~\ref{#1}}}  % Definition reference
\newcommand\Lref[1]{{Lemma~\ref{#1}}}       % Lemma reference
\newcommand\Fref[1]{{Figure~\ref{#1}}}      % Figure reference
\newcommand\eq[1]{{(\ref{#1})}}             % reference to an equation - raw
\newcommand\Magma{{\texttt{Magma}}}

\newcommand{\ubegineq}[1]{\begin{equation}\label{#1}}
\newcommand{\uendeq}{\end{equation}}

\newcommand{\uThreePointOuter}[5]{
\begin{equation}\label{#5-2}
\langle\G_{#2'},\G_{#3}\rangle=\langle\G_{#2'},\G_{#3'}\rangle=\langle\G_{#2'},
\G_{#3}^{-1}\G_{#3'}\G_{#3}\rangle=e
\end{equation}
\begin{equation}\label{#5-3}
\G_{#2}=\G_{#3'}\G_{#3}\G_{#2'}\G_{#3}^{-1}{\G_{#3'}}^{-1}
\end{equation}
\begin{equation}\label{#5-4}
\begin{split}
\langle\G_{#4},\G_{#3'}\G_{#3}\G_{#2'}\G_{#3}{\G_{#2'}}^{-1}\G_{#3}^{-1}
{\G_{#3'}}^{-1}\rangle=&
\langle\G_{#4},\G_{#3'}\G_{#3}\G_{#2'}\G_{#3'}{\G_{#2'}}^{-1}\G_{#3}^{-1}
{\G_{#3'}}^{-1}\rangle= \\
=& \langle\G_{#4},\G_{#3'}\G_{#3}\G_{#2'}\G_{#3}^{-1}\G_{#3'}\G_{#3}
{\G_{#2'}}^{-1}\G_{#3}^{-1}{\G_{#3'}}^{-1}\rangle=e
\end{split}
\end{equation}
\begin{equation}\label{#5-5}
\G_{#4'}=\G_{#4}\G_{#3'}\G_{#3}\G_{#2'}\G_{#3'}\G_{#3}{\G_{#2'}}^{-1}
\G_{#3}^{-1}{\G_{#3'}}^{-1}\G_{#4}\G_{#3'}\G_{#3}\G_{#2'}\G_{#3}^{-1}
{\G_{#3'}}^{-1}{\G_{#2'}}^{-1}\G_{#3}^{-1}{\G_{#3'}}^{-1}\G_{#4}^{-1}
\end{equation}
\begin{equation}\label{#5-6}
[\G_{#2},\G_{#4}]=[\G_{#2},\G_{#4'}]=[\G_{#2'},\G_{#4}]=
[\G_{#2'},\G_{#4'}]=e.
\end{equation}
}
\newcommand{\uFourPointInner}[6]{
	\begin{equation}\label{#6-1} \langle\G_{#2'},\G_{#3}\rangle=\langle\G_{#2'},\G_{#3'}\rangle=\langle\G_{#2'},\G_{#3}^{-1}\G_{#3'}\G_{#3}\rangle=e
	\end{equation}
	\begin{equation}\label{#6-2} \langle\G_{#4},\G_{#5}\rangle=\langle\G_{#4'},\G_{#5}\rangle=\langle\G_{#4}^{-1}\G_{#4'}\G_{#4},\G_{#5}\rangle=e
	\end{equation}
	\begin{equation}\label{#6-3}
	[\G_{#3'}\G_{#3}\G_{#2'}\G_{#3}^{-1}{\G_{#3'}}^{-1},\G_{#5}] = e
	\end{equation}
	\begin{equation}\label{#6-4}
 [\G_{#3'}\G_{#3}\G_{#2'}\G_{#3}^{-1}{\G_{#3'}}^{-1},\G_{#4}^{-1}{\G_{#4'}}^{-1}\G_{#5}^{-1}\G_{#5'}\G_{#5}\G_{#4'}\G_{#4}] = e
	\end{equation}
	\begin{equation}\label{#6-5} \langle\G_{#2},\G_{#3}\rangle=\langle\G_{#2},\G_{#3'}\rangle=\langle\G_{#2},\G_{#3}^{-1}\G_{#3'}\G_{#3}\rangle=e
	\end{equation}
	\begin{equation}\label{#6-6}
 \langle\G_{#4},\G_{#5}^{-1}\G_{#5'}\G_{#5}\rangle=\langle\G_{#4'},\G_{#5}^{-1}\G_{#5'}\G_{#5}\rangle=\langle\G_{#4}^{-1}
\G_{#4'}\G_{#4},\G_{#5}^{-1}\G_{#5'}\G_{#5}\rangle=e
	\end{equation}
	\begin{equation}\label{#6-7} [\G_{#3'}\G_{#3}\G_{#2}\G_{#3}^{-1}{\G_{#3'}}^{-1},\G_{#5}^{-1}\G_{#5'}\G_{#5}] = e
	\end{equation}
	\begin{equation}\label{#6-8}
	[\G_{#3'}\G_{#3}\G_{#2}\G_{#3}^{-1}{\G_{#3'}}^{-1}, \G_{#4}^{-1}{\G_{#4'}}^{-1}\G_{#5}^{-1}{\G_{#5'}}^{-1}\G_{#5}\G_{#5'}\G_{#5}\G_{#4'}\G_{#4}] = e
	\end{equation}
	\begin{equation}\label{#6-9}
	\G_{#3'}\G_{#3}\G_{#2'}\G_{#3}\G_{#2'}^{-1}\G_{#3}^{-1}{\G_{#3'}}^{-1} = \G_{#5}\G_{#4'}\G_{#5}^{-1}
	\end{equation}
	\begin{equation}\label{#6-10}
	\G_{#3'}\G_{#3}\G_{#2'}\G_{#3'}\G_{#2'}^{-1}\G_{#3}^{-1}{\G_{#3'}}^{-1} = \G_{#5}\G_{#4'}\G_{#4}{\G_{#4'}}^{-1}\G_{#5}^{-1}
	\end{equation}
	\begin{equation}\label{#6-11}
	\G_{#3'}\G_{#3}\G_{#2}\G_{#3}\G_{#2}^{-1}\G_{#3}^{-1}{\G_{#3'}}^{-1} = \G_{#5}^{-1}\G_{#5'}\G_{#5}\G_{#4'}\G_{#5}^{-1}{\G_{#5'}}^{-1}\G_{#5}
	\end{equation}
	\begin{equation}\label{#6-12}
	\G_{#3'}\G_{#3}\G_{#2}\G_{#3'}\G_{#2}^{-1}\G_{#3}^{-1}{\G_{#3'}}^{-1} = \G_{#5}^{-1}\G_{#5'}\G_{#5}\G_{#4'}\G_{#4}\G_{#4'}^{-1}\G_{#5}^{-1}{\G_{#5'}}^{-1}\G_{#5}.
	\end{equation}
}
\makeatletter
\newcommand{\uProjRel}[2]{
\begin{equation}\label{#1}
\G_{#2'}\G_{#2}\uProjRelChecknextarg} \newcommand{\uProjRelChecknextarg}{\@ifnextchar\bgroup{\uProjRelGobblenextarg}{ = e	\end{equation}}}
\newcommand{\uProjRelGobblenextarg}[1]{\G_{#1'}\G_{#1}\@ifnextchar\bgroup{\uProjRelGobblenextarg}{  = e. \end{equation}}}
\makeatother
\begin{document}
		\definecolor{circ_col}{rgb}{0,0,0}

\renewcommand{\subjclassname}{%

       \textup{2000} Mathematics Subject Classification}

\title[fundamental group of Galois covers of surfaces of degree $8$]{The fundamental group of Galois covers of surfaces of degree $8$}
%\footnotetext{Email address: Meirav Amram: meiravt@sce.ac.il, Cheng Gong: cgong@suda.edu.cn, Praveen Roy: praveen.roy1991@gmail.com, Uriel Sinichkin: sinichkin@mail.tau.ac.il, Uzi Vishne: vishne@math.biu.ac.il \\2020 Mathematics Subject Classification. 14D06, 14H30, 14J10. \\	{\bf Key words}: Fundamental group, degeneration, Galois cover, classification of surfaces}

%\stepcounter{footnote}
%\footnotetext{}
\author[Amram]{Meirav Amram}
\address{Meirav Amram, Department of Mathematics, Shamoon College of Engineering, Ashdod, Israel}
\email{meiravt@sce.ac.il}

\author[Gong]{Cheng Gong}
\address{Cheng Gong, Department of Mathematics, Soochow University, Suzhou 215006, Jiangsu, P. R. China}
\email{cgong@suda.edu.cn}

\author[Roy]{Praveen Kumar Roy}
\address{Praveen Kumar Roy, UM-DAE Centre for excellence in basic sciences, Mumbai, India}
\email{praveen.roy1991@gmail.com}

\author[Sinichkin]{Uriel Sinichkin}
\address{Uriel Sinichkin, School of Mathematical Sciences, Tel Aviv University, Tel Aviv, Israel}
\email{sinichkin@mail.tau.ac.il}

\author[Vishne]{Uzi Vishne}
\address{Uzi vishne, Department of Mathematics, Bar-Ilan University, Israel}
\email{vishne@math.biu.ac.il}

%\author[1]{Meirav Amram}
%\author[2]{Cheng Gong}
%\author[1,3]{Praveen Kumar Roy}
%\author[4]{Uriel Sinichkin}
%\author[5]{Uzi Vishne}

%\affil[1]{\small{Department of Mathematics, Shamoon College of Engineering, Ashdod, Israel}}
%\affil[2]{\small{Department of Mathematics, Soochow University, Suzhou 215006, Jiangsu, P. R. China}}
%\affil[3]{\small{UM-DAE Centre for excellence in basic sciences, Mumbai, India}}
%\affil[4]{\small{School of Mathematical Sciences, Tel Aviv University, Tel Aviv, Israel}}
%\affil[5]{\small{Department of Mathematics, Bar-Ilan University, Israel}}
	
	%\date{\today}
	
	%\maketitle
	
\begin{abstract}
We compute the fundamental group of the Galois cover of a surface of degree $8$, with singularities of degree $4$ whose degeneration is homeomorphic to a sphere. The group is shown to be a metabelian group of order $2^{23}$. The computation amalgamates local groups, classified elsewhere, by an iterative combination of computational and group theoretic methods. Three simplified surfaces, for which the fundamental group of the Galois cover is trivial, hint toward complications that depend on the homotopy of the degenerated surface.
\end{abstract}


%\date{\today}

\maketitle

%	
%	\tableofcontents
%	
	
\section{Introduction}\label{outline}
	
The classification of algebraic surfaces is a classical problem in algebraic geometry.
The fundamental group of the complement of a branch curve and the fundamental group of the Galois cover of the surface, are  important invariants of the surface. Degeneration is a crucial tool in classifying surfaces, leading to a presentation of these groups by generators and relations.
Some recent studies of degenerations and their structure appear in \cite{CCFM,2008,CLM}.

Given a surface $X$ embedded in~$\CP^N$ for some~$N$, and a generic projection to~$\C\P^2$ with a branch curve~$S$, the fundamental group $G=\pi_1(\CP^2-S)$ is an invariant of the degeneration class of $X$.
The standard generators in the Van Kampen presentation of this group are denoted by $\gamma_1,\gamma_{1'},\dots,\gamma_q,\gamma_{q'}$, where~$q$ is the number of lines in the degenerated object. There is a natural map from the quotient~$G/\nsg{\gamma^2}$ to the symmetric group $S_n$, where $n = \deg X$ is the degree of the projection and~$\nsg{\gamma^2}$ is the  normal subgroup generated by $\gamma_1^2,\gamma_{1'}^2,\dots,\gamma_q^2,\gamma_{q'}^2$. The fundamental group of the Galois cover $\Galois{X}$ is the kernel of this map \cite{Li08}, so we have a short exact sequence
$$1\longrightarrow \pi_1(\Galois{X}) \longrightarrow \pi_1(\CP^2-S)/\nsg{\gamma^2} \longrightarrow S_n \longrightarrow 1.$$

The presentation of group $G$ is fairly complicated and it is often a challenging task to ascertain basic properties of this group. The computation has been carried out in many cases, for example, for $\C\P^1 \times \C\P^1$ \cite{10}, the Veronese surfaces \cite{14,15}, Hirzebruch surfaces $F_k(a,b)$ for $a$ and $b$ relatively prime \cite{13}, $\C\P^1 \times T$ \cite{AG,CP1T} where $T$ is the complex torus, toric varieties~\cite{Ogata}, and surfaces with Zappatic singularities of type~$E_k$ \cite{ZAPP}. In one of the most complex cases, that of the surface $T \times T$ \cite{TT},  the Van Kampen presentation had 54 generators and nearly 2000 relations.

To the degeneration of a surface $X$ we associate the {\bf{shadow}}, which is a $2$-dimensional simplicial complex. There is one $2$-cell (i.e., a triangle) for each irreducible component of the degeneration, which are all copies of $\CP^2$. Two triangles share a $1$-cell (an edge), respectively a $0$-cell (a vertex), if the corresponding planes intersect in a line, respectively at a point. In particular, the number of triangles in the shadow is the degree of~$X$. Edges that belong to a single triangle are {\bf{redundant}}. The number of non-redundant edges  meeting in a vertex is the multiplicity (or the degree) of the singularity at this point. The shadow is {\bf{planar}} if it can be embedded in the Euclidean plane~$\R^2$. Experience shows that the computation of the fundamental group is typically easier when the shadow is planar.

Most examples studied so far have planar shadows \cite{ZAPP,Ogata,13,14,15}.  Relatively new works on surfaces of degrees $5$ and $6$, with planar shadows, appear in \cite {degree5,degree6}. Recently, the first named author studied surfaces with non-planar shadows that have degrees~$4$ and~$6$, see \cite{Nonplanar1}. The shadow of the surface $\CP^1 \times T$ (where $T$ is the torus), which was studied in \cite{AG}, is homeomorphic to a cylinder.

In this work we focus on surfaces of degree $8$, whose shadow is homeomorphic to the sphere $S^2$, such that every vertex in the shadow has degree~$4$. Such a shadow can be described as the union of two triangulated squares glued along all four edges, see Figure \ref{X4};
for that reason we denote the shadow as $X^4$.
As expected, $\pi_1(\Galois{X^4})$  has rather complicated computations. We thus implement the following technique:
We introduce intermediate shadows, obtained by ``cutting open'' one glued pair of edges in each step. There are thus four shadows,
$$X^1 \rightarrow X^2 \rightarrow X^3 \rightarrow X^4,$$
where each arrow denotes a covering map obtained from identifying two edges, see Figures~\ref{X1}, \ref{X2}, and~\ref{X3}, respectively.
Computing the four fundamental groups in this manner becomes a sequence of problems of increasing complexity, each of which requires additional ideas to solve.
For $i = 1,2,3$, we were able to use \Magma\ to  simplify the presentation of $\pi_1(\Galois{X^i})$,
resulting in each case in a trivial group. Singularities of high multiplicity complicate the computation, so for $X^4$ we needed  to incorporate the methods of \cite{ALV} and \cite{RTV}.

For every graph $T$ there are a Coxeter group $C(T)$ and a natural quotient $C_Y(T)$.
This is relevant here because in all cases computed so far, the fundamental group of the surface turns out to be a quotient of $C_Y(T)$, where $T$ is the dual graph of the shadow. Naively, we would want to make all computations in $C_Y(T)$, because this group is fully understood.
However, not all the defining relations of $C_Y(T)$ are given {\it a priori}. This is resolved by a technical innovation of a multi-layered computation. We run \Magma\ on the given presentation. When we cannot show that the fundamental group is indeed a quotient of $C_Y(T)$, we show that a subgroup is a quotient of $C_Y(T')$ for a subgraph $T'$, use this to obtain more relations, run \Magma\ again, and repeat.
This is best demonstrated in the computation for $X^3$ below.

As we progress from $X^1$ to $X^2$ to $X^3$ and finally to $X^4$, the situation becomes more complex. For $X^1$ the traditional techniques suffice. For $X^2$ we used \Magma\ and methods from  \cite{RTV}. For $X^3$ we used the multi-layered computation described above. For $X^4$ we needed to also combine ideas from \cite{ALV}, which computes the Artin covers of $C_Y(T)$.

The remainder of this paper is organized as follows. In Section \ref{section:method} we give the constructions of $X^i$, $i=1,2,3,4$, we recall the Van Kampen Theorem to get presentations of fundamental groups, and give some necessary details on Coxeter groups and dual graphs from the papers \cite{ALV} and \cite{RTV}.
In Theorems \ref{PropX1}, \ref{PropX2}, \ref{PropX3}, and \ref{PropX4}  (Section \ref{results})  the groups $\pi_1(\Galois{X})$ related to  $X^i$, $i=1,2,3,4$, are determined.

\forget
((We note that degenerations with planar representations of a bounded degree can be enumerated recursively, because the planar simplicial complexes can be enumerated. In the current paper, it is unlikely to obtain such easy enumerations, and much work and thought should be invested to fit the numbers and produce a geometric object that matches the constrains of the work process. Moreover, some of the methods are no longer the same traditional ones they were until now. For example, determination of groups is no longer carried out in the usual way, but we are already reaching the construction of algebraic computer programs.))
\forgotten

%\paragraph{Acknowledgments:}

\section{Preliminaries and methods}\label{section:method}	

In this section we describe the shadows $X^1$, $X^2$, $X^3$, and $X^4$ in detail, and apply standard techniques to exhibit  presentations for their fundamental groups $\pi_1(\Galois{X^i})$. We recommend  \cite{CCFM,2008,CLM} as sources for relevant information about degenerations and \cite{degree6} for detailed background about fundamental groups of complements of branch curves and fundamental groups of Galois covers.

%We construct now the four degree 8 degenerations and give their figures. Then we explain in short the algorithm of finding groups $G$ and $\pi_1(\Galois{X})$, and then determine $\pi_1(\Galois{X})$ for each Galois covers of $X^i$, \ $i=1,2,3,4$.

\subsection{Projective degeneration}

We begin by defining a degeneration.
Let $\Delta$ denote the unit disc, and let $X, Y$ be projective algebraic surfaces.
		Let $p: Y \rightarrow \mathbb{CP}^2$ and $p': X \rightarrow \mathbb{CP}^2$
be generic projections.
We say that $p'$ is a projective degeneration of $p$
		if there is a flat family $\pi: V \rightarrow \Delta$
		and an embedding $F:V\rightarrow \Delta \times \mathbb{CP}^2$,
		such that $F$ composed with the first projection is $\pi$,
		and:
		\begin{itemize}
			\item[(a)] $\pi^{-1}(0) \simeq X$;
			\item[(b)] there is a $t_0 \neq 0$ in $\Delta$ such that
			$\pi^{-1}(t_0) \simeq Y$;
			\item[(c)] the family $V-\pi^{-1}(0) \rightarrow \Delta-{0}$
			is smooth;
			\item[(d)] restricting to $\pi^{-1}(0)$, $F = {0}\times p'$
			under the identification of $\pi^{-1}(0)$ with $X$;
			\item[(e)] restricting to $\pi^{-1}(t_0)$, $F = {t_0}\times p$
			under the identification of $\pi^{-1}(t_0)$ with $Y$.
		\end{itemize}
In the above process, we construct a flat degeneration of $ X $ into a union of planes. We consider degenerations with only two planes intersecting at a line, with each plane homeomorphic to $\mathbb{CP}^2$.

\subsection{Constructions and figures of degenerations}

Here we describe the four shadows $X^i$, $i = 1,2,3,4$, which as explained above are all of degree $8$ and have singularities of degree at most (and in $X^4$: precisely) $4$.

\begin{construction}[The shadow~$X^4$]

A triangulation of degree $8$ of the sphere is obtained by cutting the Euclidean sphere by three orthogonal planes passing through the origin, see \Fref{X4}. This triangulation has $6$ vertices which we denote $V_1,\dots,V_6$ and $12$ edges; the repeated edges, $5,6,8,9$, are glued in pairs.
We numerate the vertices and edges for future reference.

\forget
\begin{figure}[H]
  \begin{center}
    \definecolor{circ_col}{rgb}{0,0,0}
      \begin{tikzpicture}[x=1cm, y=1cm, scale=1]

    		\draw [fill=circ_col] (-4,0.5) node [above] {} circle (1pt);
    		\draw [fill=circ_col] (-2,0.1) node [below] {} circle (1pt);
    		\draw [fill=circ_col] (-3,2) node [above] {} circle (1pt);
    		\draw [fill=circ_col] (-3.4,-0.3) node [below] {} circle (1pt);
    		\draw [fill=circ_col] (-2.6,0.8) node [below] {} circle (1pt);

		% \draw [-, line width = 1pt] (-4,0.5) -- (-2,0.1);
		\draw [-, line width = 1pt, dashed] (-4,0.5) -- (-2.6,0.8) -- (-2,0.1);
		\draw [-, line width = 1pt, dashed] (-2.6,0.8) -- (-3,2);		
		\draw [-, line width = 1pt] (-4,0.5) -- (-3.4,-0.3);
		\draw [-, line width = 1pt] (-2,0.1) -- (-3.4,-0.3);
		\draw [-, line width = 1pt] (-2,0.1) -- (-3,2);
		\draw [-, line width = 1pt] (-4,0.5) -- (-3,2);
		\draw [-, line width = 1pt] (-3.4,-0.3) -- (-3,2);

                % \draw [-, line width = 1pt] (1,0) -- node [anchor=east]{$ 1 $} (1,1)

                \draw [<-|, line width =1pt] (-1.3,1) -- (-1.7,1);

\newcommand{\radiusx}{1}
\newcommand{\radiusy}{.3}
\newcommand{\height}{2}

\coordinate (a) at (-{\radiusx*sqrt(1-(\radiusy/\height)*(\radiusy/\height))},{\radiusy*(\radiusy/\height)});

\coordinate (b) at ({\radiusx*sqrt(1-(\radiusy/\height)*(\radiusy/\height))},{\radiusy*(\radiusy/\height)});

\draw[fill=gray!30, line width =1pt] (a)--(0,\height);
\draw[fill=gray!30, line width =1pt] (b)--(0,\height);

%\fill[gray!50] circle (\radiusx{} and \radiusy);

\begin{scope}
\clip ([xshift=-2mm]a) rectangle ($(b)+(1mm,-2*\radiusy)$);
\draw [line width =1pt] circle (\radiusx{} and \radiusy);
\end{scope}

\begin{scope}
\clip ([xshift=-2mm]a) rectangle ($(b)+(1mm,2*\radiusy)$);
\draw [line width =1pt] circle (\radiusx{} and \radiusy);
\end{scope}

\draw[line width =1pt] (0,\height)|-(0, -0.30); %(\radiusx,0);% node[right, pos=.25]{}; % node[above,pos=.75]{$r$};

%\draw (0,.15)-|(.15,0);
\end{tikzpicture}

\end{center}
\caption{Deformation of degree $4$ degeneration to a cone over $E_0$.}\label{fig_cone}
\end{figure}

The intersection of two such cones gives a deformation of~$X^4$ as described in Figure \ref{fig_double_cone}.

\begin{figure}[H]
\begin{center}
\begin{tikzpicture}[x=1cm, y=1cm, scale=1]
\newcommand{\radiusx}{1}
\newcommand{\radiusy}{.3}
\newcommand{\height}{2}

\coordinate (a) at (-{\radiusx*sqrt(1-(\radiusy/\height)*(\radiusy/\height))},{\radiusy*(\radiusy/\height)});

\coordinate (b) at ({\radiusx*sqrt(1-(\radiusy/\height)*(\radiusy/\height))},{\radiusy*(\radiusy/\height)});

\coordinate (c) at (-{\radiusx*sqrt(1-(\radiusy/\height)*(\radiusy/\height))},{\radiusy*(\radiusy/\height)*-1});
\coordinate (d) at ({\radiusx*sqrt(1-(\radiusy/\height)*(\radiusy/\height))},{\radiusy*(\radiusy/\height)*-1});

\draw[fill=gray!30, line width =1pt] (a)--(0,\height);
\draw[fill=gray!30, line width =1pt] (b)--(0,\height);
\draw[fill=gray!30, line width =1pt] (c)--(0,-\height);
\draw[fill=gray!30, line width =1pt] (d)--(0,-\height);

%\fill[gray!50] circle (\radiusx{} and \radiusy);

\begin{scope}
\clip ([xshift=-2mm]a) rectangle ($(b)+(1mm,-2*\radiusy)$);
\draw [line width =1pt] circle (\radiusx{} and \radiusy);
\end{scope}

\begin{scope}
\clip ([xshift=-2mm]a) rectangle ($(b)+(1mm,2*\radiusy)$);
\draw [line width =1pt] circle (\radiusx{} and \radiusy);
\end{scope}

\draw[line width =1pt] (0,\height)|-(0, -0.30); %(\radiusx,0);% node[right, pos=.25]{}; % node[above,pos=.75]{$r$};
\draw[line width =1pt] (0,-\height)|-(0, -0.30);
%\draw (0,.15)-|(.15,0);
\end{tikzpicture}
\end{center}
\caption{Two cones intersecting along $E_0$.}\label{fig_double_cone}
\end{figure}
\forgotten

\begin{figure}[H]
$$\xymatrix@R=30pt@C=30pt{
{}  & V_5 \ar@{-}[dl]|9 \ar@{-}[d]|2 \ar@{-}[dr]|5 & {} & {} & V_5 \ar@{-}[dl]|5 \ar@{-}[d]|{10} \ar@{-}[dr]|9 & {} \\
V_6 \ar@{-}[r]|1 \ar@{-}[dr]|8 & V_1 \ar@{-}[r]|4 \ar@{-}[d]|3 & V_3 \ar@{-}[dl]|6 & V_3 \ar@{-}[r]|7 \ar@{-}[dr]|6 & V_4 \ar@{-}[r]|{12} \ar@{-}[d]|{11} & V_6 \ar@{-}[dl]|8 \\
{}  & V_2 & {} & {} & V_2 & {}
}$$
\caption{Degree $4$ degenerations with four identified edges give $X^4$}\label{X4}
\end{figure}
\end{construction}

It is worth noting that
following \cite[Corollary 12.2]{2008}, each degree~$4$ degeneration can be deformed to a cone over an elliptic curve $E_0$, where $E_0$ is a curve of degree~$4$ in $\mathbb{P}^3$.
%, as depicted in Figure \ref{fig_cone}.
The  Zappatic singularities  of type~$E_4$ are obtained as the intersection of~$4$ lines. For more details about such surfaces,  see \cite{ZAPP}.

Now consider the two connected components of \Fref{X4}, without the gluing. We will construct $X^1$, $X^2$, and $X^3$ by gluing one pair, then two pairs, and eventually, three pairs of edges.

\begin{construction}[The shadow~$X^1$]

The shadow $X^1$ appears in Figure \ref{X1}. It has~$8$ vertices and~$15$ edges, of which only~$9$ edges are non-redundant (the redundant edges, those belonging to a single triangle, are not numerated, as they do not appear in the computation. Edge $5$ appears twice in the figure, and the two copies are of course glued.

Notice that the degree of $V_1$ and $V_2$ is $4$; the degree of $V_3$ and $V_4$ is $3$; and the other vertices have degree $1$.
\end{construction}

\begin{figure}[H]
$$\xymatrix@R=30pt@C=30pt{
{}  & V_3 \ar@{-}[dl]|{} \ar@{-}[d]|2 \ar@{-}[dr]|5 & {} & {} & V_3 \ar@{-}[dl]|5 \ar@{-}[d]|{7} \ar@{-}[dr]|{} & {} \\
V_5 \ar@{-}[r]|1 \ar@{-}[dr]|{} & V_1 \ar@{-}[r]|4 \ar@{-}[d]|3 & V_4 \ar@{-}[dl]|{} & V_4 \ar@{-}[r]|6 \ar@{-}[dr]|{} & V_2 \ar@{-}[r]|{9} \ar@{-}[d]|{8} & V_8 \ar@{-}[dl]|{} \\
{}  & V_6 & {} & {} & V_7 & {}
}$$
	\caption{Degree $4$ degenerations with one identified edge give $X^1$}\label{X1}
\end{figure}

We now glue in $X^1$ the edge connecting the vertices $V_4$ and $V_6$ with the edge connecting the vertices $V_4$ and $V_7$, to obtain $X^2$ (the enumeration of the vertices and edges changes).

\begin{construction}[Degeneration $X^2$]

The degeneration $X^2$ appears in Figure \ref{X2}.
This shadow has $7$ vertices and $10$ non-redundant edges. Edges $5$ and $6$ appear twice, and the two copies are being glued. The degree of $V_1$, $V_2$, and $V_3$ is $4$; the degree of $V_4$ and $V_5$ is $3$; and the degree of $V_6$ and $V_7$ is $1$.
\begin{figure}[H]
$$\xymatrix@R=30pt@C=30pt{
{}  & V_4 \ar@{-}[dl]|{} \ar@{-}[d]|2 \ar@{-}[dr]|5 & {} & {} & V_4 \ar@{-}[dl]|5 \ar@{-}[d]|{8} \ar@{-}[dr]|{} & {} \\
V_6 \ar@{-}[r]|1 \ar@{-}[dr]|{} & V_1 \ar@{-}[r]|4 \ar@{-}[d]|{3} & V_2 \ar@{-}[dl]|{6} & V_2 \ar@{-}[r]|7 \ar@{-}[dr]|{6} & V_3 \ar@{-}[r]|{10} \ar@{-}[d]|{9} & V_7 \ar@{-}[dl]|{} \\
{}  & V_5 & {} & {} & V_5 & {}
}$$
	\caption{Degree $4$ degenerations with two identified edges give $X^2$}\label{X2}
\end{figure}
\end{construction}

We now glue the edge connecting $V_5$ and $V_6$ with the edge connecting $V_5$ and $V_7$ (in the numeration of $X^2$) to obtain $X^3$.

\begin{construction}[The shadow $X^3$]
There are now $6$ vertices and $11$ non-redundant edges, see \Fref{X3}.
The degrees of $V_1$, $V_2$, $V_3$, and $V_4$ are $4$; and the degrees of $V_5$ and $V_6$ are $3$.

\begin{figure}[H]
$$\xymatrix@R=30pt@C=30pt{
{}  & V_5 \ar@{-}[dl]|{} \ar@{-}[d]|2 \ar@{-}[dr]|5 & {} & {} & V_5 \ar@{-}[dl]|5 \ar@{-}[d]|{9} \ar@{-}[dr]|{} & {} \\
V_6 \ar@{-}[r]|1 \ar@{-}[dr]|8 & V_1 \ar@{-}[r]|4 \ar@{-}[d]|3 & V_3 \ar@{-}[dl]|6 & V_3 \ar@{-}[r]|7 \ar@{-}[dr]|6 & V_4 \ar@{-}[r]|{11} \ar@{-}[d]|{10} & V_6 \ar@{-}[dl]|8 \\
{}  & V_2 & {} & {} & V_2 & {}
}$$
	\caption{Degree $4$ degenerations with three identified edges give $X^3$}\label{X3}
\end{figure}
\end{construction}

One can also depict $X^1$, $X^2$, and $X^3$ as planar shadows, without gluing, as shown in \Fref{X123}. We maintain the numeration of the vertices and edges for each shadow as given above. To obtain $X^4$ from $X^3$ we must glue the two edges connecting $V_5$ and $V_6$ (one from the triangle $V_5V_6V_4$ and one from $V_5V_6V_1$), and the resulting shadow is no longer planar.

\begin{figure}[H]
$$
\xymatrix@R=20pt@C=20pt{
{} & {} & V_8 \ar@{-}[dl]
\ar@{-}[d]|{9}
\ar@{-}[dr]
& {} \\
{} & V_3 \ar@{-}[r]|{7}
\ar@{-}[dl]
\ar@{-}[d]|{2}
\ar@{-}[dr]|{5}
& V_2 \ar@{-}[r]|{8} \ar@{-}[d]|{6} & V_7 \ar@{-}[dl] \\
V_5 \ar@{-}[r]|{1} \ar@{-}[dr]& V_1 \ar@{-}[d]|{3} \ar@{-}[r]|{4} & V_4 \ar@{-}[dl] & {} \\
{} & V_6& {} & {}
}
\qquad
\xymatrix@R=20pt@C=20pt{
{} & {} & V_7 \ar@{-}[dl]
\ar@{-}[d]|{10}
\ar@/^4ex/@{-}[dddr]
& {} \\
{} & V_4 \ar@{-}[r]|{8}
\ar@{-}[dl]
\ar@{-}[d]|{2}
\ar@{-}[dr]|{5}
& V_3 \ar@{-}[d]|{7} \ar@/^2ex/@{-}[ddr]|{9} & {}
 \\
V_6 \ar@{-}[r]|{1} \ar@/_4ex/@{-}[drrr]& V_1 \ar@/_2ex/@{-}[drr]|{3} \ar@{-}[r]|{4} & V_2 \ar@{-}[dr]|{6} & {} \\
{} & {} & {} & V_5
}
\qquad
\xymatrix@R=20pt@C=20pt{
V_5 \ar@{-}[r]|{9}
\ar@{-}[d]|{2}
\ar@{-}[dr]|{5}
& V_4 \ar@{-}[d]|{7} \ar@/^2ex/@{-}[ddr]|{10} & {}
&{} \\
V_1 \ar@/_2ex/@{-}[drr]|{3} \ar@{-}[r]|{4} & V_3 \ar@{-}[dr]|{6} & {} &{} \\
{} & {} & V_2 &{} \\
{} & {} & {} &
V_6
\ar@{-}[ul]|{8}
\ar@/^14ex/@{-}[uuulll]
\ar@/_14ex/@{-}[uuulll]
\ar@/_4ex/@{-}[uuull]|{11}
\ar@/^4ex/@{-}[uulll]|{1}
}
\forget
\qquad
\xymatrix@R=20pt@C=20pt{
V_6
\ar@/^10ex/@{-}[dddrrr]|{8}
\ar@/_10ex/@{-}[dddrrr]|{8}
\ar@/^2ex/@{-}[drr]|{11}
\ar@/_1ex/@{-}[dr]
\ar@/^1ex/@{-}[dr]
\ar@/_2ex/@{-}[ddr]|{1}
& {} & {} & {}
\\
 {} & V_5 \ar@{-}[r]|{9}
\ar@{-}[d]|{2}
\ar@{-}[dr]|{5}
& V_4 \ar@{-}[d]|{7} \ar@/^2ex/@{-}[ddr]|{10} & {}
 \\
{} & V_1 \ar@/_2ex/@{-}[drr]|{3} \ar@{-}[r]|{4} & V_3 \ar@{-}[dr]|{6} & {} \\
{} & {} & {} & V_2
}
\forgotten
$$
\caption{Alternative presentations for the shadows $X^1$, $X^2$ and $X^3$.}\label{X123}
\end{figure}

\subsection{Fundamental groups}

We now describe how the  fundamental groups for  $X^i$, $i=1,2,3,4$, are computed from the diagrams. Let $X$ be an  algebraic surface of degree~$n$ embedded in some projective space, with a generic projection $f \co X \rightarrow \C\P^2$ of degree $n$. Let
$$\Galois{X}=\overline{(X \times_{\mathbb{CP}^{2}} \ldots \times_{\mathbb{CP}^{2}} X)-\triangle},$$
be the Galois cover of $X$, where the product is taken $n$ times, and $\triangle$ is the extended diagonal that is defined as
$$\triangle=\{(x_1,\dots,x_k)\in X^k\bigm| x_i=x_j\quad\text{for some}\quad i\ne j\}.$$

The projection has a branch curve $S$, and we turn to compute  $G=\pi_1(\mathbb{CP}^2-S)$ and $\pi_1(\Galois{X})$.
The degeneration of $X$ into a union of planes has a branch curve which is a line arrangement of degree $q$, the number of edges in the degeneration.

The {\bf{vertices}} of the line arrangement are the points where more than two lines intersect. The multiplicity of a vertex is the number of lines $k$ that intersect at the point. In this work, we have singularities with multiplicities $1$, $3$, and $4$.
There is a regeneration process (explained, for example, in \cite{degree6}), which ``opens'' each of the lines to either a double line or a conic. We obtain a local description for the regenerated curve~$S$, which is a cuspidal curve of degree $2q$.
To compute the presentation we first must choose a global ordering of the edges in the shadow. Then, the relations arising from each vertex  depend on the local ordering at this point.
Now, the Van Kampen Theorem \cite{VK} states that group $G = \pi_1(\CP^2-S)$ is generated by $\G_i, \G_{i'}$ for $i=1, \dots, q$, where each generator represents an edge between two planes in the degeneration, subject to relations of the following five types.

\begin{enumerate}
\item for every branch point of a conic, a relation of the form $\gamma = \gamma'$ where $\gamma,\gamma'$ are certain conjugates of the generators $\G_{j}$ and $\G_{j'}$,
\item\label{vK2} for every node,
$[\gamma,\gamma'] = 1$ where $\gamma,\gamma'$ are certain conjugates of $\gamma_i$ and $\gamma_j$, where $i,j$ are the lines meeting in this node,
\item for every cusp, $\langle\gamma,\gamma'\rangle=1$ where $\gamma,\gamma'$ are as in \eq{vK2},
\item\label{vK4} the ``projective relation'' $\prodl_{i=d}^1  \G_{i'}\G_i = 1$,
\item[(5*)] a commutator relation for every pair of edges which meet in the branch curve.
\end{enumerate}

As a consequence, other than the projective relation, the relations can be computed locally. In other words, $G$ is an amalgamated product of ``local groups'', one for each vertex, generated by the generators corresponding to the edges meeting in the vertex --- modulo the projective relation of \eq{vK4} and the relations of type (5*). The amalgamation is done by identifying the generators associated to the same edge, each appearing in the two local groups of the vertices of this edge.

The following lemmas summarize the local groups for vertices of multiplicity $3$ or $4$ after regeneration, for a particular order of the edges. Our numeration of the edges in each of the shadows $X^i$ was chosen so that locally we only see the orderings appearing in these lemmas.

In Figure \ref{fig_lemma3pt_outer} we have an intersection of three lines, being globally ordered as $ i < j < k $.
\begin{figure}[H]
\[
\xymatrix@R=30pt@C=30pt{
{} &{} & \\
{} &*={\bullet} \ar@{}[r] \ar@{}[l] \ar@{-}[ur]|i \ar@{-}[ul]|k \ar@{-}[u]|j&
}
\]
	\caption{Intersection of three lines}\label{fig_lemma3pt_outer}
\end{figure}

\begin{lemma}\label{Lemma3pt}
For each vertex whose numeration is as in \Fref{fig_lemma3pt_outer}, the relations induced on $G$ are:
\uThreePointOuter{p}{i}{j}{k}{3pt-rels}
\end{lemma}

Figure \ref{4pt} depicts an intersection of four lines, being globally ordered as $ i < j < k < l$.
\begin{figure}[H]
\[
\xymatrix@R=30pt@C=30pt{
{} &{} & \\
{} &*={\bullet} \ar@{-}[r]|k \ar@{-}[l]|j \ar@{-}[d]|i \ar@{-}[u]|{\,l\,} & \\
{} &{} &
}
\]
	\caption{Intersection of four lines}\label{4pt}
\end{figure}

\begin{lemma}\label{Lemma4pt}
A point that is an intersection of four lines contributes to $G$ the following list of relations:
\uFourPointInner{q}{i}{j}{k}{l}{4pt-rels}
\end{lemma}

\subsection{The dual graph}\label{ss:CYT}

Group $G$ comes with a natural projection to the symmetric group $S_n$, where each generator $\gamma_i$ or $\gamma_{i'}$ is mapped to the transposition switching the planes meeting in edge number~$i$.

{}From a presentation of $G$ we can find a presentation of $\pi_1(\Galois{X})$, using the exact sequence
\begin{equation}\label{Gal}
	0 \rightarrow \pi_1(\Galois{X}) \rightarrow
G/\nsg{\gamma^2} \rightarrow S_n \rightarrow 0;
	\end{equation}
this can be done using the Reidmeister-Schreier method; however, the number of generators gets multiplied by $n!$, so one must  find ways to make the computation manageable.

Consider a shadow, such as $X^i$ for $i = 1,2,3,4$. We now define a map from the  fundamental group $G/\nsg{\gamma^2}$ to the symmetric group $S_n$, where $n$ is the number of $2$-cells in the shadow. Recall that $G/\nsg{\gamma^2}$ is generated by  $\gamma_i$ and $\gamma_{i'}$, one pair for each nonredundant edge of the shadow. The map to $S_n$ is then defined by mapping both $\gamma_i$ and $\gamma_{i'}$ to the transposition $(ab)$, where $a,b$ are the triangles meeting in the common edge $i$.

A set of transpositions in $S_n$ can be encoded as the edges of a graph $T$ on $n$ vertices. The transpositions generate the whole group if and only if $T$ is connected. Furthermore, $S_n$ can be presented on the given transpositions with respect to a certain set of relations. Group $C_Y(T)$, which we define below, is defined by imposing the ``local'' relations, and so there is a projection $C_Y(T) \rightarrow S_n$. In all cases studied so far, it was possible to prove that the generators of $G/\nsg{\gamma^2}$ satisfy the defining relations of  $C_Y(T)$.

The {\bf{dual graph}} of the shadow is graph $T$, whose vertices are the $2$-cells, with an edge connecting every intersecting pair of $2$-cells. For example, the dual graph of $X^1$ (from \Fref{X1}), is given in \Fref{dualX1}.

Because the product of two  transpositions is of order $2$ or $3$, depending on whether they intersect, we are led to associate to the dual graph $T$ an abstract group $C_Y(T)$ defined as follows.

\begin{Def}\label{def:c_y_t}
Let $T$ be a connected graph, not necessarily simple, but without loops. Let $C(T)$ be the Coxeter group whose generators are the edges of $T$, subject to three types of relations: $u^2 = 1$ for every generator; $(uv)^2=1$ if the edges corresponding to $u,v$ are disjoint; and $(uv)^3 = 1$ if $u,v$ intersect in one vertex. (No relation is assumed if $u,v$ connect the same two vertices).

Furthermore, we define the group $C_Y(T)$ as the quotient of $C(T)$ with respect to four types of relations:
\begin{eqnarray}
{}[w u w, v] & = & e
    \qquad \mbox{if $u,v,w$ are as in Figure \ref{0knifeX} (left)}, \label{CYT1}\\
{} \langle w u w,{v} \rangle & = & e
    \qquad \mbox{if $u,v,w$ are as in Figure \ref{0knifeX} (right)},
    \label{CYT2}\\
{}[w u w, v x v] & = & e
    \qquad \mbox{if $x,u,v,w$ are as in Figure \ref{0plateX} (left)},
    \label{CYT3} \\
{}\langle w u w ,v xv \rangle & = & e
    \qquad \mbox{if $x,u,v,w$ are as in Figure \ref{0plateX} (right)}
% \textcolor{red}{why \ do \ we \ need \ this \ one?}. --- because we want to be consistent with other situations where this relation is needed.
\label{CYT4}
\end{eqnarray}

\begin{figure}%[!h]
\begin{equation*}
\xy
    (1,0);(16,0)    **\crv{\dir{-}} ?(0)*\dir{*}             +(8,2)*{u};
    (16,0);(28,10)  **\crv{\dir{-}} ?(0)*\dir{*}?(1)*\dir{*} +(-7,-3)*{v};
    (16,0);(28,-10) **\crv{\dir{-}} ?(1)*\dir{*}             +(-7,3)*{w};
\endxy
\qquad
\xy
    (0,0);(16,0)    **\crv{\dir{-}} ?(0)*\dir{*}             +(8,2)*{u};
    (16,0);(32,0)  **\crv{(24,10)} ?(0)*\dir{*}?(1)*\dir{*} +(-8,7)*{v};
    (16,0);(32,0)  **\crv{(24,-10)} ?(0)*\dir{*}?(1)*\dir{*} +(-8,-7)*{w};
 \endxy
\end{equation*}
\caption{}\label{0knifeX}
\end{figure}

\begin{figure}[!h]
\begin{equation*}
\xy
    (4,0);(16,0)    **\crv{\dir{-}} ?(0)*\dir{*}             +(6,2)*{u};
    (16,0);(32,0)  **\crv{(24,10)} ?(0)*\dir{*}?(1)*\dir{*} +(-8,7)*{v};
    (16,0);(32,0)  **\crv{(24,-10)} ?(0)*\dir{*}?(1)*\dir{*} +(-8,-7)*{w};
    (32,0);(44,0) **\crv{\dir{-}} ?(1)*\dir{*}             +(-6,2)*{x};
\endxy
\qquad
\xy
    (16,0);(32,0)  **\crv{(24,10)} ?(0)*\dir{*}?(.5)*\dir{*}?(1)*\dir{*} +(-3,5)*{x} +(-10,0)*{u};
    (16,0);(32,0)  **\crv{(24,-10)} ?(0)*\dir{*}?(1)*\dir{*} +(-8,-7)*{w};
    (16,0);(32,0) **\crv{\dir{-}} ?(1)*\dir{*}             +(-8,1.5)*{v};
\endxy
\end{equation*}
\caption{}\label{0plateX}
\end{figure}
\end{Def}

It is easy to see that if we map each generator to the transposition, switching its two vertices, all the relations defining $C_Y(T)$ hold in $S_n$. We thus have a well-defined projection $C_Y(T) \rightarrow S_n$, where $n$ is the number of vertices.

It is not hard to see that if $T$ is a tree, then $C_Y(T)$ is isomorphic to $S_n$. In general, the transpositions arranged in a cycle satisfy another relation, the ``cyclic relation'', which does not hold in~$C_Y(T)$. This was used to compute the kernel of the map $C_Y(T) \rightarrow S_n$, as we now explain.

\subsection{The group $S_n \ltimes A_{t,n}$}\label{Vishne}

The kernel of the map $C_Y(T) \rightarrow S_n$ was computed in \cite{RTV}, when $T$ is a simple graph. For a planar graph $T$, this was extended in \cite{ALV} to a similarly defined group $A_Y(T)$, which naturally covers Artin's braid group $B_n$; the kernel of the map $A_Y(T) \rightarrow B_n$ was computed there. From this result one can deduce the structure of the kernel of $C_Y(T) \rightarrow S_n$ for any planar graph $T$, simple or not. We believe that the same description holds for any graph (see \cite{NC} for details), but this is not needed for the current paper.
% MSc thesis of Noa Cohen, Bar Ilan University, forthcoming

Let $t \geq 0$ and $n$ be natural numbers. Recall the definition of group $A_{t,n}$ from \cite{RTV}. Let $U : = \{u, u',\dots \}$ be a set consisting of $t$ elements.

\begin{Def}
Group $A_{t,n}$ is generated by the $n^2\card{U}$ elements $u_{x,y}$, for $u \in U$ and $x,y \in \set{1,\dots,n}$, satisfying the following relations for any $u,u' \in U$ and any $x,y,z$:
 \begin{enumerate}
 \item $u_{x,x} = e$
 \item $u_{x,y}u_{y,z} = u_{x,z} = u_{y,z}u_{x,y}$
 \item $[u_{x,y},u_{w,z}'] = e$ \quad \textit{(for distinct $x,y,w$ and $z$)}
 \end{enumerate}
\end{Def}

Another description of the same group is provided by the following.
\begin{Def}\label{Ftndef}
For $x \in \set{1,\dots,n}$, let $F^{(x)}$ denote the free group on $t$ letters $u_x$ $(u \in U)$. Put $F_{t,n}^* := \prod\limits_{x\in \set{1,\dots,n}} F^{(x)}$. Map $\textit{ab} : F^{(x)} \rightarrow \mathbb{Z}^t$, defined by $\textit{ab}(u_x) = u$, (where $\mathbb{Z}^t$ is thought of as a free Abelian group generated by $U$), can be extended naturally to give map $F_{t,n}^* \rightarrow \mathbb{Z}^t$. We define the kernel of this map to be $F_{t,n}$ \cite[Definition 5.6]{RTV}.
\end{Def}

The main result of \cite{RTV} is that  $C_Y(T) \cong S_n \ltimes A_{t,n}$ where $n$ is the number of vertices in $T$ and $t$ is the rank of $\pi_1(T)$. The symmetric group acts on $A_{t,n}$, by its action on the indices. It is also verified in \cite[Theorem~5.7]{RTV} for $n \geq 5$, that $A_{t,n} \cong F_{t,n}$. In particular, there is a short exact sequence $$1 \longrightarrow F_{t,n} \longrightarrow C_Y(T) \longrightarrow S_n \longrightarrow 1.$$

\section{Results}\label{results}
In this section we describe the groups $\pi_1(\Galois{X^i})$ for each $i=1,2,3,4$.
We apply the algorithm and methods described in Section \ref{section:method}.

%In this section we do NOT repeat the methods from the previous section. We try to be as short as possible.

We have the following notation that will be used in commutation relations.
\begin{notation}
A formula involving  $\G_{i,i'}$ is a shorthand for all the formulas that are obtained by replacing this term with $\G_i$ or $\G_{i'}$.
\end{notation}
\forget
In the following definition we give a convention for sets of indices in order to write commutations in a shorter manner.
\begin{Def}\label{short}
\[
\{i_1,i_2,\dots,i_k\}_{j_1,j_2,\dots,j_m} := \{{i_l}_{j_t}: l =1,\dots,k\ \text{and}\ t = 1,\dots, m\}
\]
and,
\begin{eqnarray*}
\{\{i_1,i_2,\dots,i_{k_1}\}_{j_1,j_2,\dots,j_{m_1}}, \{i_{k_1+1},\dots,i_{k_2}\}_{j_{m_1+1},\dots,j_{m_2}}, \dots,
\{i_{k_{n-1}+1},\dots,i_{k_n}\}_{j_{m_{n-1}+1},\dots,j_{m_n}}\} \cr
= \cup_{p=0}^{n-1} \{i_{k_p+1},\dots,i_{k_{p+1}}\}_{j_{m_p+1},\dots,j_{m_{p+1}}},
\end{eqnarray*}
where $k_0$ and $m_0$ we define to be $0$.
\end{Def}
\forgotten

\subsection{The surface $X^1$}

\begin{thm}\label{PropX1}
The fundamental group of the Galois cover of $X^1$ is trivial.
\end{thm}
\begin{proof}
Let us describe group $G/\nsg{\gamma^2}$ for the shadow $X^1$ depicted in  Figure \ref{X1}. The group is generated by $\set{\G_i}$ for $i=1, 1', 2, 2', \dots, 9, 9'$. Vertices $V_1$ and $V_2$ give relations as in Lemma \ref{Lemma4pt}, where $(i,j,k,l)$ are equal to $(1,2,3,4)$ and $(6,7,8,9)$, respectively.
For the vertices $V_3$ and $V_4$ we have the relations from Lemma \ref{Lemma3pt}, with
$(i,j,k)$ equal to $(2,5,7)$ and $(4,5,6)$,   respectively.
The vertices $V_5$, $V_6$, $V_7$, and $V_8$ give the relations $\G_i=\G_{i'}$ for $i=1, 3, 8$, and $9$ respectively. Then we have the relations of type (5*):
$$[\G_{i,i'}, \G_{j,j'}]=e \ \ \mbox{for} \ \  (i,j) \in \Set{1,3}{5,6,7,8,9}\cup \Set{2}{6,8,9} \cup \Set{4}{7,8,9} \cup \Set{5}{8,9}.$$
The last relation is the projective one $\prodl_{i=9}^1  \G_{i'}\G_i = e$.

Next we compare (\ref{4pt-rels-9}) and (\ref{4pt-rels-10}) with $(i,j,k,l)$ equal to $(1,2,3,4)$, and together with $\G_1=\G_{1'}$ and $\G_3=\G_{3'}$, we get $\G_2=\G_{2'}$. Then we can get from (\ref{4pt-rels-9}) and (\ref{4pt-rels-11}) that $\G_4=\G_{4'}$. Similarly, we use $(i,j,k,l)$ equal to $(6,7,8,9)$, along with $\G_8=\G_{8'}$ and $\G_9=\G_{9'}$ to get $\G_7=\G_{7'}$ and $\G_6=\G_{6'}$ as well.

Now we turn to relation (\ref{3pt-rels-3}) and substitute $i=2$ and $j=5$ along with $\G_2=\G_{2'}$, which was obtained before, to get  $\G_5=\G_{5'}$. In a similar way, we substitute $i=4$ and $j=5$ and $\G_5=\G_{5'}$, and get $\G_4=\G_{4'}$.

It follows that group $G/\nsg{\gamma^2}$ is generated by $\set{\G_i}$, $1 \leq i \leq 9$.
We now present the dual graph of the edges in $X^1$, see Figure \ref{dualX1}, from which one can easily read the map $G/\nsg{\gamma^2} \rightarrow S_8$ (where $8$ is the number of vertices in the dual graph).

\begin{figure}[H]
$$\xymatrix@R=30pt@C=30pt{
{} & *={\bullet} \ar@{-}[rd]|2 & {} & {} & *={\bullet} \ar@{-}[rd]|9 & {} \\
*={\bullet} \ar@{-}[ru]|1 \ar@{-}[rd]|3 & {} & *={\bullet} \ar@{-}[r]|5 & *={\bullet} \ar@{-}[ru]|7 \ar@{-}[rd]|6 & {} & *={\bullet} \\
{} & *={\bullet} \ar@{-}[ru]|4 & {} & {} & *={\bullet} \ar@{-}[ru]|8 & {}
}$$
\caption{The dual graph $T_1$ related to $X^1$}\label{dualX1}
\end{figure}

\forget
\begin{figure}[H]
	\begin{center}
		
		\definecolor{circ_col}{rgb}{0,0,0}
		\begin{tikzpicture}[x=1cm,y=1cm,scale=2]
		
		\draw [fill=circ_col] (-1,0) node [below] {} circle (1pt);
		\draw [fill=circ_col] (-2,1) node [anchor=east] {} circle (1pt);
		\draw [fill=circ_col] (-1,2) node [above] {} circle (1pt);
		\draw [fill=circ_col] (0,1) node [anchor=west] {} circle (1pt);
		%\draw [fill=circ_col] (-1,1) node [anchor=north west] {$ V_1 $} circle (1pt);

                 \draw [fill=circ_col] (2,0) node [below] {} circle (1pt);
		\draw [fill=circ_col] (1,1) node [anchor=east] {} circle (1pt);
		\draw [fill=circ_col] (2,2) node [above] {} circle (1pt);
		\draw [fill=circ_col] (3,1) node [anchor=west] {} circle (1pt);
		%\draw [fill=circ_col] (2,1) node [anchor=north west] {$ V_2 $} circle (1pt);

		\draw [-, line width = 1pt] (-1,0) -- (-2,1);
		\draw [-, line width = 1pt] (-1,0) -- (0,1);
		%\draw [-, line width = 1pt] (-1,0) -- (-1,2);
		%\draw [-, line width = 1pt] (-2,1) -- (0,1);
		\draw [-, line width = 1pt] (-2,1) -- (-1,2);
		\draw [-, line width = 1pt] (-1,2) -- (0,1);
		
		\draw [-, line width = 1pt] (2,0) -- (1,1);
		\draw [-, line width = 1pt] (2,0) -- (3,1);
		%\draw [-, line width = 1pt] (2,0) -- (2,2);
		%\draw [-, line width = 1pt] (1,1) -- (3,1);
		\draw [-, line width = 1pt] (1,1) -- (2,2);
		\draw [-, line width = 1pt] (2,2) -- (3,1);
		
		\draw [-, line width = 1pt] (0,1) -- (1,1);
		
		\draw [-, line width = 1pt] (-1,0) -- node [below]{$ 4 $} (0,1);
		%\draw [-, line width = 1pt] (-1,0) -- node [anchor=east]{$ 3 $} (-1,1);
		\draw [-, line width = 1pt] (-1,0) -- node [below]{$ 3 $} (-2,1);
		\draw [-, line width = 1pt] (-2,1) -- node [above]{$ 1 $} (-1,2);
		\draw [-, line width = 1pt] (-1,2) -- node [above]{$ 2 $} (0,1);
		%\draw [-, line width = 1pt] (-1,2) -- node [anchor=east]{$ 2 $} (-1,1);
		%\draw [-, line width = 1pt] (-1,1) -- node [below]{$ 1 $} (-2,1);
		%\draw [-, line width = 1pt] (0,1) -- node [below]{$ 4 $} (-1,1);
		
		\draw [-, line width = 1pt] (2,0) -- node [below]{$ 8 $} (3,1);
		%\draw [-, line width = 1pt] (2,0) -- node [anchor=east]{$ 6 $} (2,1);
		\draw [-, line width = 1pt] (2,0) -- node [below]{$ 6 $} (1,1);
		\draw [-, line width = 1pt] (1,1) -- node [above]{$ 7 $} (2,2);
		\draw [-, line width = 1pt] (2,2) -- node [above]{$ 9 $} (3,1);
		%\draw [-, line width = 1pt] (2,2) -- node [anchor=east]{$ 7 $} (2,1);
		%\draw [-, line width = 1pt] (2,1) -- node [below]{$ 6 $} (1,1);
		%\draw [-, line width = 1pt] (2,1) -- node [below]{$ 9 $} (3,1);
		
		\draw [-, line width = 1pt] (0,1) -- node [below]{$ 5 $} (1,1);

		\end{tikzpicture}
		
	\end{center}
	\setlength{\abovecaptionskip}{-0.15cm}
	\caption{The dual graph related to $X^1$}\label{dualX1}
\end{figure}
\forgotten
The relations that we have in $G/\nsg{\gamma^2}$, from the above simplification, are of type
\begin{center}
\begin{eqnarray}
  [\G_i,\G_j]=e, \ \  \mbox{if i,j are disjoint lines;}\label{com} \\
  \langle \G_i,\G_j \rangle=e, \ \  \mbox{otherwise.}\nonumber
\end{eqnarray}
\end{center}
We also have two relations $\G_1\G_2 \G_1=\G_3 \G_4\G_3$ and $\G_6\G_7 \G_6=\G_8 \G_9\G_8$ that relate to the two cycles in the dual graph. Moreover, relations $[\G_4,\G_2 \G_5\G_2]=e$ and $[\G_5,\G_6 \G_7\G_6]=e$ appear in our presentation, and they are associated to the two sets of three lines each meeting at a point in the graph. This means that $G/\nsg{\gamma^2} \rightarrow S_8$ is an isomorphism, and since $\pi_1(\Galois{X^1})$ is the kernel of this map, the fundamental group of the Galois cover of $X^1$ is trivial.
\end{proof}

\subsection{The surface $X^2$}

\begin{thm}\label{PropX2}
The fundamental group of the Galois cover related to $X^2$ is trivial.
\end{thm}

\begin{proof}
The group $G/\nsg{\gamma^2}$ is generated by $\set{\G_i,\G_{i'}}$ for $i=1, \dots, 10$. Following  Figure \ref{X2}, vertices $V_1$, $V_2$ and $V_3$ give the relations from Lemma \ref{Lemma4pt}, where $(i,j,k,l)$ equal to $(1,2,3,4)$ for $V_1$, to $(4,5,6,7)$ for $V_2$, and to $(7,8,9,10)$ for $V_3$.
Vertices $V_4$ and $V_5$ give the relations from Lemma \ref{Lemma3pt}, with $(i,j,k)$ equal to $(2,5,8)$ and $(3,6,9)$, respectively.
The vertices $V_6$ and $V_7$ give the relations $\G_i=\G_{i'}$ for $i=1$ and $i=10$, respectively.
The relations of type (5*) are:
$$[\G_{i,i'}, \G_{j,j'}]=e \ \ \mbox{for} \ \  (i,j) \in \begin{array}{l} \Set{1}{5,6,7,8,9,10}\cup \Set{2}{6,7,9,10}\cup\Set{3}{5,7,8,10} \\
\cup \Set{4}{8,9,10}\cup\Set{5}{9,10}\cup\Set{6}{8,10}.
\end{array}$$
And as always, the projective relation is  $\prodl_{i=10}^1  \G_{i'}\G_i = e$.

Now we explain the simplification of the presentation of $G/\nsg{\gamma^2}$.
As written above, we have 20 generators and a quite a few relations. We try a run of the presentation in \Magma\ and get as an output a simplified presentation of $G/\nsg{\gamma^2}$ with eight generators $\G_1, \G_{3'}, \G_5, \G_{5'}, \G_{6'}, \G_7, \G_9, \G_{10}$, still with a long list of relations.

We construct the dual graph according to the list of generators, see Figure \ref{twoedges-firstdual}.

\begin{figure}[H]
$$\xymatrix@R=30pt@C=30pt{
{} & {} & {} & {} & {} &*={\bullet} \ar@{-}[r]|5  &*={\bullet} \ar@{-}[l]<4pt>^{5'} \\
{} &*={\bullet} \ar@{-}[r]|1 &*={\bullet} \ar@{-}[r]|{3'} &*={\bullet}  \ar@{-}[r]|{6'} &*={\bullet}  \ar@{-}[ru]|7 \ar@{-}[rd]|9  \\
{} & {} & {} & {} & {} &*={\bullet} \ar@{-}[r]|{10} &*={\bullet}
}$$
\caption{The dual graph after the first run in \Magma}\label{twoedges-firstdual}
\end{figure}

\forget
\begin{figure}[H]
	\begin{center}
		
		\definecolor{circ_col}{rgb}{0,0,0}
		\begin{tikzpicture}[x=1cm,y=1cm,scale=2]
		
		\draw [fill=circ_col] (0,0) node [below] {} circle (1pt);
		\draw [fill=circ_col] (1,0) node [below] {} circle (1pt);
		\draw [fill=circ_col] (2,0) node [below] {} circle (1pt);
		\draw [fill=circ_col] (3,0) node [below] {} circle (1pt);
		\draw [fill=circ_col] (4,1) node [below] {} circle (1pt);
		\draw [fill=circ_col] (5,1) node [below] {} circle (1pt);
		\draw [fill=circ_col] (4,-1) node [below] {} circle (1pt);
		\draw [fill=circ_col] (5,-1) node [below] {} circle (1pt);
		
		\draw [-, line width = 1pt] (0,0) -- (3,0);
		\draw [-, line width = 1pt] (3,0) -- (4,1);
		\draw [-, line width = 1pt] (4,1) -- (5,1);
		\draw [-, line width = 1pt] (4.1,0.9) -- (4.9,0.9);
		\draw [-, line width = 1pt] (3,0) -- (4,-1);
		\draw [-, line width = 1pt] (4,-1) -- (5,-1);
		
		\draw [-, line width = 1pt] (0,0) -- node [below]{$ 1 $} (1,0);
		\draw [-, line width = 1pt] (1,0) -- node [below]{$ 3' $} (2,0);
		\draw [-, line width = 1pt] (2,0) -- node [below]{$ 6' $} (3,0);
		\draw [-, line width = 1pt] (3,0) -- node [above]{$ 7 $} (4,1);
		\draw [-, line width = 1pt] (4,1) -- node [above]{$ 5 $} (5,1);
		\draw [-, line width = 1pt] (4.1,0.9) -- node [below]{$ 5' $} (4.9,0.9);
		\draw [-, line width = 1pt] (3,0) -- node [below]{$ 9 $} (4,-1);
		\draw [-, line width = 1pt] (4,-1) -- node [below]{$ 10 $} (5,-1);
		
		\end{tikzpicture}
		
	\end{center}
	\setlength{\abovecaptionskip}{-0.15cm}
	\caption{The dual graph after the first run in \Magma}\label{twoedges-firstdual}
\end{figure}
\forgotten
We see in the output that there is a missing relation $[\G_7,\G_{10}]=e$, which can be derived from another relation $[\G_1 \G_7 \G_{6'} \G_{3'} \G_{6'} \G_1 \G_7, \G_{10}]=e$ in the output. We then add it to the presentation, and run \Magma\ for the second time.

The output of the second run on \Magma\  consists of $\G_1, \G_2, \G_{3'}, \G_{5'}, \G_{6'}, \G_7, \G_9, \G_{10}$, as generators, again with  a long list of relations. We pick up all relations of type (\ref{com}) from the simplified presentation, and construct the dual graph related to these generators, see Figure \ref{twoedges-seconddual}.

\begin{figure}[H]
$$\xymatrix@R=30pt@C=30pt{
{} e\ar@{-}[d]^1 \ar@{-}[r]^2  &f \ar@{-}[r]^{5'}  &a \ar@{-}[d]^7 & \\
{} d \ar@{-}[r]^{3'}  &c  \ar@{-}[r]^{6'} &b  \ar@{-}[r]^{9} &g  \ar@{-}[r]^{10} &h
}$$
\caption{The dual graph after the second run in \Magma}\label{twoedges-seconddual}
\end{figure}

\forget
\begin{figure}[H]
	\begin{center}
	\definecolor{circ_col}{rgb}{0,0,0}
\begin{tikzpicture}[x=1cm,y=1cm,scale=2]
\draw [fill=circ_col] (0,0) node [below] {d} circle (1pt);
\draw [fill=circ_col] (1,0) node [below] {c} circle (1pt);
\draw [fill=circ_col] (0,1) node [above] {e} circle (1pt);
\draw [fill=circ_col] (1,1) node [above] {f} circle (1pt);
\draw [fill=circ_col] (2,1) node [above] {a} circle (1pt);
\draw [fill=circ_col] (2,0) node [below] {b} circle (1pt);
\draw [fill=circ_col] (3,0) node [below] {g} circle (1pt);
\draw [fill=circ_col] (4,0) node [below] {h} circle (1pt);
\draw [-, line width = 1pt] (0,0) -- (0,1);
		\draw [-, line width = 1pt] (0,0) -- (1,0);
		\draw [-, line width = 1pt] (0,1) -- (1,1);
		\draw [-, line width = 1pt] (1,1) -- (2,1);
		\draw [-, line width = 1pt] (1,0) -- (2,0);
		\draw [-, line width = 1pt] (2,1) -- (2,0);
		\draw [-, line width = 1pt] (2,0) -- (3,0);
		\draw [-, line width = 1pt] (3,0) -- (4,0);
         \draw [-, line width = 1pt] (0,0) -- node [anchor=east]{$ 1 $} (0,1);
\draw [-, line width = 1pt] (0,0) -- node [below]{$ 3' $} (1,0);
\draw [-, line width = 1pt] (0,1) -- node [above]{$ 2 $} (1,1);
\draw [-, line width = 1pt] (1,1) -- node [above]{$ 5' $} (2,1);
                 \draw [-, line width = 1pt] (1,0) -- node [below]{$ 6' $} (2,0);
\draw [-, line width = 1pt] (2,0) -- node [anchor=west]{$ 7 $} (2,1);
\draw [-, line width = 1pt] (2,0) -- node [below]{$ 9 $} (3,0);
\draw [-, line width = 1pt] (3,0) -- node [below]{$ 10 $} (4,0);
\end{tikzpicture}
		
	\end{center}
	\setlength{\abovecaptionskip}{-0.15cm}
	\caption{The dual graph after the second run in \Magma}\label{twoedges-seconddual}
\end{figure}
\forgotten
In the output we could not find the relation $[\G_{5'},\G_{6'}]=e$, but the relation $[\G_{6'}\G_{3'} \G_{6'},\G_{5'}]=e$ appearing in the output simplifies to $[\G_{5'},\G_{6'}]=e$.
At this point, a cyclic relation for the cycle appearing in the graph, is still missing.

For $T$, the graph depicted in Figure~\ref{twoedges-seconddual}, we get a well-defined surjection $ \varphi:C_Y(T)\to G/\nsg{\G^2} $, see Definition \ref{def:c_y_t}.
This map sends the generator corresponding to an edge $ E $ to the generator $ \G_i $, where $ i $ is the label on edge $ E $ in Figure~\ref{twoedges-seconddual}.
It is well-defined by the definition of $ C_Y(T) $.
By~\cite{RTV}, the map $ \psi:C_Y(T)\to S_8 \ltimes A_{1,8}$, defined by $$\G_1 \mapsto (e\ d), \ \ \G_2 \mapsto (e\ f), \ \ \G_{3'} \mapsto (c\ d), \ \ \G_{5'} \mapsto (a\ f)u_{a,f}, \ \ \G_{6'} \mapsto (b\ c),$$
 and
 $$ \G_7 \mapsto (a\ b), \ \ \G_9 \mapsto (b\ g), \ \ \G_{10} \mapsto (g\ h)$$
	is an isomorphism.
	Here, $S_8$ acts on the label set $L = \set{a,b,c,d,e,f,g,h}$, and $A_{1,8}$ is generated by the elements of the form $u_{x,y}$ (see Subsection \ref{Vishne}), for $x,y \in L$.
	So, we are left to compute $ \ker \varphi\circ \psi^{-1} $.
	This is done by pulling back the defining relations of $ G $ to $ S_8 \ltimes A_{1,8} $, which gives us a unique non-vanishing relation: $u_{h,f} u_{g,d} = e$.
However, we can act with $S_8$ on this relation, and obtain $u_{x,y} = u_{x',y'}^{-1}$ for any distinct $x,y,x',y' \in L$. In particular, $u_{a,b}^2 = u_{c,d}u_{d,e} = u_{c,e} = u_{a,b}$, so $u_{a,b} = 1$ and all the generators of $A_{1,8}$ are contained in $\ker \varphi\circ\psi^{-1}$.

We thus proved that  $G/\nsg{\gamma^2}\cong S_8$. In particular, the fundamental group of $\Galois{X^2}$ is trivial.
\end{proof}

\subsection{The surface $X^3$}

\begin{thm}\label{PropX3}
The fundamental group of the Galois cover related to $X^3$ is trivial.
\end{thm}

\begin{proof}
Group $G/\nsg{\gamma^2}$ related to $X^3$ is generated by $\set{\G_i}$ for $i=1, 1', 2, 2', \dots, 11, 11'$. Vertices $V_1$, $V_2$, $V_3$, and $V_4$ give the relations from Lemma \ref{Lemma4pt}, where $(i,j,k,l)$ are equal to $(1,2,3,4)$ for $V_1$, $(3,6,8,10)$ for $V_2$,  $(4,5,6,7)$ for $V_3$, and $(7,9,10,11)$ for $V_4$.
Vertices $V_5$ and $V_6$ give the relations from Lemma \ref{Lemma3pt}, with $(i,j,k)$ equals $(2,5,9)$ and $(1,8,11)$, respectively.
The relations of type (5*) are:
$$[\G_{i,i'}, \G_{j,j'}]=e$$ for $$(i,j) \in \begin{array}{l}
\Set{1}{5,6,7,9,10} \cup  \Set{2}{6,7,8,10,11}\cup\Set{3}{5,7,9,11}\cup \Set{4}{9} \\
\cup \Set{4,5}{8,10,11} \cup  \Set{6}{9,11} \cup \Set{8}{7,9}, \end{array}
$$
and the projective relation is  $\prodl_{i=11}^1  \G_{i'}\G_i = e$.

Group $G/\nsg{\gamma^2}$ now has $22$ generators and a long list of relations. This input was fed into \Magma, and the output is now the group with the generators $\G_{1'}, \G_{2'}, \G_5, \G_{5'}, \G_{6'}, \G_7, \G_{8'}, \G_9$, and $\G_{10}$, and a long list of relations. Let us label the triangles of $X^3$, as in \Fref{X3++}.

\begin{figure}[H]
$$\xymatrix@R=30pt@C=30pt{
{}  & V_5 \ar@{-}[dl]|{} \ar@{-}[d]|2 \ar@{-}[dr]|5 & {} & {} & V_5 \ar@{-}[dl]|5 \ar@{-}[d]|{9} \ar@{-}[dr]|{} & {} \\
V_6 \ar@{-}[r]|1 \ar@{-}[dr]|8 & V_1
\ar@{}[ru]|(0.3){\textcircled{a}}
\ar@{}[lu]|(0.3){\textcircled{b}}
\ar@{}[ld]|(0.3){\textcircled{c}}
\ar@{}[rd]|(0.3){\textcircled{h}}
\ar@{-}[r]|4 \ar@{-}[d]|3 & V_3 \ar@{-}[dl]|6 & V_3 \ar@{-}[r]|7 \ar@{-}[dr]|6 & V_4
\ar@{}[ru]|(0.3){\textcircled{g}}
\ar@{}[lu]|(0.3){\textcircled{f}}
\ar@{}[ld]|(0.3){\textcircled{e}}
\ar@{}[rd]|(0.3){\textcircled{d}}
\ar@{-}[r]|{11} \ar@{-}[d]|{10} & V_6 \ar@{-}[dl]|8 \\
{}  & V_2 & {} & {} & V_2 & {}
}$$
	\caption{The shadow $X^3$ with triangles labelled}\label{X3++}
\end{figure}

The dual graph $T$ with these generators is now given in \Fref{threeedges-firstdual}.
\begin{figure}[H]
\[
\xymatrix@R=30pt@C=30pt{
{}  &b \ar@{-}[r]^{2'}  &a \ar@/^1ex/@{-}[r]^5 \ar@/_1ex/@{-}[r]<-4pt>_{5'} &f \ar@{-}[d]^7  \ar@{-}[r]^9 &g & \\
{}  &c\ar@{-}[r]^{8'} \ar@{-}[u]^{1'}  &d \ar@{-}[r]^{10} &e \ar@{-}[r]^{6'} &h
}
\]
\caption{The dual graph for the remaining generators}\label{threeedges-firstdual}
\end{figure}
\forget
\begin{figure}[H]
	\begin{center}
		
		\definecolor{circ_col}{rgb}{0,0,0}
		\begin{tikzpicture}[x=1cm,y=1cm,scale=2]
		
		\draw [fill=circ_col] (0,0) node [below] {d} circle (1pt);
		\draw [fill=circ_col] (0,1) node [above] {c} circle (1pt);
		\draw [fill=circ_col] (1,1) node [above] {b} circle (1pt);
		\draw [fill=circ_col] (2,1) node [above] {a} circle (1pt);
		\draw [fill=circ_col] (3,0) node [below] {e} circle (1pt);
		\draw [fill=circ_col] (4,0) node [below] {h} circle (1pt);
		\draw [fill=circ_col] (3,1) node [above] {f} circle (1pt);
		\draw [fill=circ_col] (4,1) node [above] {g} circle (1pt);
		
		\draw [-, line width = 1pt] (0,0) -- (3,0);
		\draw [-, line width = 1pt] (0,1) -- (1,1);
		\draw [-, line width = 1pt] (1,1) -- (2,1);
		\draw [-, line width = 1pt] (2,1) -- (3,1);
		\draw [-, line width = 1pt] (2.1,0.9) -- (2.9,0.9);
		\draw [-, line width = 1pt] (3,1) -- (4,1);
		\draw [-, line width = 1pt] (2,0) -- (3,0);
		\draw [-, line width = 1pt] (3,0) -- (4,0);

                 \draw [-, line width = 1pt] (0,0) -- node [anchor=east]{$ 8' $} (0,1);
                 \draw [-, line width = 1pt] (0,0) -- node [below]{$ 10 $} (3,0);
                 \draw [-, line width = 1pt] (0,1) -- node [above]{$ 1' $} (1,1);
                 \draw [-, line width = 1pt] (1,1) -- node [above]{$ 2' $} (2,1);
                 \draw [-, line width = 1pt] (2,1) -- node [above]{$ 5 $} (3,1);
                 \draw [-, line width = 1pt] (3,1) -- node [anchor=west]{$ 7 $} (3,0);
                 \draw [-, line width = 1pt] (3,1) -- node [above]{$ 9 $} (4,1);
                 \draw [-, line width = 1pt] (3,0) -- node [below]{$ 6' $} (4,0);
                 \draw [-, line width = 1pt] (2.1,0.9) -- node [below]{$ 5' $} (2.9,0.9);

		\end{tikzpicture}
		
	\end{center}
	\setlength{\abovecaptionskip}{-0.15cm}
	\caption{The dual curve after  running \Magma}\label{threeedges-firstdual}
\end{figure}
\forgotten
The list of simplified relations we got from \Magma\ contains all the defining relations of $C_Y(T)$ of a type \eqref{com} corresponding to $T$ except $\langle \G_5, \G_{7}\rangle=e, \langle \G_5, \G_{9}\rangle=e, \langle \G_{5'}, \G_{7}\rangle=e, \langle \G_{6'}, \G_{10}\rangle=e$ and $\langle \G_7, \G_{10}\rangle=e$.
It also contains the relations $\langle\G_{6'}\G_{7}\G_{6'},\G_{5'}\rangle=e$, $\langle\G_{6'}\G_{5'}\G_{8'}\G_{6'},\G_{7}\G_{10}\G_{8'}\G_7\rangle=e$, and $\langle\G_{6'}\G_{8'}\G_{7}\G_{6'}\G_{8'},\G_{1'}\G_{5}\G_{2'}\G_{5}\G_{1'}\rangle=e$, which can be simplified to $\langle\G_{5'},\G_{7}\rangle=e$, $\langle\G_{7},\G_{10}\rangle=e$, and $\langle \G_5, \G_{7}\rangle=e$ respectively.
We thus obtained the relations of type \eqref{com} corresponding to $T$, except $\langle \G_5, \G_{9}\rangle=e$ and $\langle \G_{6'}, \G_{10}\rangle=e$.

Denote by $T_0'$ the subgraph of $T$ resulting by removing the edges labeled $6'$, $9$, and $5'$, and by~$T_0$ the subgraph of $T$ resulting by removing the edges labeled $6'$, $9$, and $5$.
Since we have derived the relations corresponding to $T_0$ and $T_0'$, we conclude that there are natural maps $C_Y(T_0)\to G/\nsg{\gamma^2}$ and $C_Y(T_0')\to G/\nsg{\gamma^2}$.
Those maps agree on the generators that appear in both groups, i.e., $ \gamma_{1'},\gamma_{2'},\gamma_7,\gamma_{8'},\gamma_{10}$, and so we get a map from their amalgamated product. Each of the graphs $T_0$ and $T_0'$ has one cycle, so by \cite{RTV}, $C_Y(T_0)$ and $C_Y(T_0')$ are both isomorphic to $S_6 \ltimes A_{1,6}$ by
$$\G_{1'} \mapsto (b\ c), \G_{2'} \mapsto (a\ b), \G_{5} \mapsto (a\ f)u_{a,f}, \G_{5'} \mapsto (a\ f)u_{a,f}', \G_7 \mapsto (e\ f), \G_{8'} \mapsto (c\ d), \G_{10} \mapsto (d\ e).$$
Here, $S_6$ acts on the label set $L = \set{a,b,c,d,e,f}$, the copy of $A_{1,6}$ corresponding to $C_Y(T_0)$ is generated by the elements of the form $u_{x,y}$ for $x,y \in L$, and the copy of $A_{1,6}$ corresponding to $C_Y(T_0')$ is generated by elements of the form $u'_{x,y}$ for $x,y\in L$  (see Subsection \ref{Vishne}).
Thus, their amalgamated product is isomorphic to $S_6 \ltimes (A_{1,6} * A_{1,6})$, which we denote by $H$.
The resulting map to $G/\nsg{\G^2}$ is of course not surjective, so we take the free product with the missing generators to get the map
$$\varphi_1 : H * \langle \G_{6'}, \G_9 \rangle \twoheadrightarrow G/\nsg{\G^2},$$
of which we need to find the kernel after adding the remaining relations.

One of the relations in the output pulls back by $\varphi_1$ to
\begin{eqnarray*}
(a\ b\ c)\G_{6'}(a\ b)(c\ d\ e)\G_{6'}(a\ b\ c)(d\ e)\G_{6'}(a\ b)(c\ d\ e)\G_{6'}(a\ b\ d\ c)\G_{6'}(a\ b)(c\ e\ d) = e.
\end{eqnarray*}
Using the fact that $ \G_{6'}$ commutes with all the permutations that fix $ e $, this relation simplifies to $ \left\langle \G_{6'}, \left( d \; e \right) \right\rangle = e $ which is the pullback (by $ \varphi_1 $) of $\langle\G_{6'},\G_{10}\rangle = e$, one of the missing relations of type~\eqref{com}.

Adding edge $6'$ to $ T_0 $ and $ T_0' $ and performing the procedure described above by counting the newly derived relation $\langle\G_{6'},\G_{10}\rangle = e$, we get a map
$$\varphi_2 : (S_7\ltimes (A_{1,7} * A_{1,7})) * \langle \G_9 \rangle \twoheadrightarrow G/\nsg{\G^2}.$$
Here, $S_7$ acts on the label set $L = \set{a,b,c,d,e,f,h}$; one of the copies of $A_{1,7}$ is generated by the elements of the form $u_{x,y}$ and the other by elements of the form $u_{x,y}'$ for $x,y \in L$.

Another relation in the output pulls back by $\varphi_2$ to $ u_{f,a}u_{c,d}'u_{a,f}u_{d,c}' = e $ which, after conjugating with transpositions, implies that $u_{x,y}$ commutes with $u_{w,z}'$ for all distinct $x,y$ and $w,z$ (recall that $ u_{y,x}=u_{x,y}^{-1} $).

Next we consider the relation from the output that pulls back by $\varphi_2$ to
\begin{eqnarray*}
u_{e,a}'u_{h,e}u_{a,h}'u_{c,a}u_{h,c}'u_{a,h}u_{c,e}' = e,
\end{eqnarray*}
and simplify it. We use $\sim$ to denote conjugacy in the group.
\begin{eqnarray*}
u_{c,e} & \sim & u_{c,h}'u_{c,e}u_{h,c}' \\
& = & u_{c,e}'u_{e,h}'u_{c,a}u_{a,e}u_{h,c}' \\
& \sim & u_{e,h}'u_{c,a}u_{h,c}'u_{a,e}u_{c,e}' \\
& = & u_{e,h}'u_{c,a}u_{h,c}'u_{a,h}u_{d,e}u_{h,d}u_{c,e}' \\
& \sim & u_{h,d}u_{e,h}'u_{c,a}u_{h,c}'u_{a,h}u_{d,e}u_{c,e}' \\
& = &
u_{h,d}u_{e,a}'u_{a,h}'u_{c,a}u_{h,c}'u_{a,h}u_{d,e}u_{c,e}' \\
& = & u_{e,a}'u_{h,d}u_{d,e}u_{a,h}'u_{c,a}u_{h,c}'u_{a,h}u_{c,e}'  =  e.
\end{eqnarray*}

This implies that $u_{x,y} = e$ for all distinct $x$ and $y$.
In particular, $\G_5$ pulls back by $ \varphi_2 $ to $ (a\; f) $, i.e., $ \G_5 $ is a conjugation on $ \G_7 $ by $ \G_{2'}\G_{1'}\G_{8'}\G_{10} $.
Since $ \G_{1'}, \G_{2'}, \G_{8'} $, and $ \G_{10} $ commute with $ \G_9 $, and since $ \langle \G_7, \G_{9}\rangle = e $ holds, we conclude that $ \langle \G_5, \G_9 \rangle = e $ holds as well.

Moreover, $ \G_5 $ can be expressed as product of other generators.
This means that we have a surjection $ C_Y(\hat{T})\twoheadrightarrow G/\nsg{\G^2} $ where $ \hat{T} $ is the subgraph of $ T $ that we get by removing edge $ 5 $.
We want to find the kernel of this surjection.

As before, $ C_Y(\hat{T}) $ is isomorphic to $ S_8 \ltimes A_{1,8} $ by
$$\G_{1'} \mapsto (b\ c), \ \ \G_{2'} \mapsto (a\ b), \ \ \G_{5'} \mapsto (a\ f)u_{a,f}', \ \ \G_{6'} \mapsto (e\ h), \ \
\G_7 \mapsto (e\ f),$$
and
$$ \G_{8'} \mapsto (c\ d), \ \ \G_9 \mapsto (f\ g), \ \ \G_{10} \mapsto (d\ e).$$
Pulling back some of the defining relations of $ G/\nsg{\G^2} $ to $ C_Y(\hat{T}) $ and using the above isomorphism with $ S_8\ltimes A_{1,8} $, we get the relations $u_{f,d}'^2 = e$ and $u_{c,g}'u_{c,d}'u_{f,d}'u_{e,a}'u_{h,b}' = e$.
We can use them to find that $u_{w,z}' = e$ for all $w,z$. To this end, we first conjugate $u_{f,d}'^2 = e$ by a transposition  $(f\ c)$ to get $u_{c,d}'^2 = e$, and then simplify the other relation:
\begin{equation}\label{eq:X_3_u_prime_is_trivial1}
	e = u_{c,g}'u_{c,d}'u_{f,d}'u_{e,a}'u_{h,b}' = u_{d,g}'u_{c,d}'^2u_{f,d}'u_{e,a}'u_{h,b}' = u_{f,g}'u_{e,a}'u_{h,b}'.
\end{equation}
Conjugating by the transposition $(h\ c)$ we get
\begin{equation}\label{eq:X_3_u_prime_is_trivial2}
u_{f,g}'u_{e,a}'u_{c,b}' = e.
\end{equation}
The difference of \eqref{eq:X_3_u_prime_is_trivial1} and \eqref{eq:X_3_u_prime_is_trivial2} is $ u_{h,c}' = e $, and after conjugating by elements of $ S_8 $ we get $u_{w,z}' = e$ for all $w,z$. This implies that $A_{1,8}$ is contained in the kernel we are computing, and therefore $G/\nsg{\gamma^2}\cong S_8$. This means that the fundamental group of the Galois cover related to $X^3$ is trivial.
\end{proof}

\subsection{The surface $X^4$}

Finally, we consider a surface $X$ whose degeneration is the shadow $X^4$.
Recall that $G = \pi_1(\CP^2-S)$ is the fundamental group of the complement of the branch curve. We let $\bar{G}$ denote $\pi_1(\Galois{X}) = G/\nsg{\gamma^2}$.

\begin{thm}\label{PropX4}
The fundamental group $\bar{G}$ of the Galois cover related to $X^4$ is an abelian-by-abelian group of order $2^{23}$.
(In fact, a central extension $1 \rightarrow \Z_2^3 \rightarrow \bar{G} \rightarrow \Z_2^{20} \rightarrow 1$.)
\end{thm}

\begin{proof}
Group $G/\nsg{\gamma^2}$, in this case, is generated by $\set{\G_i}$ for $i=1, 1', 2, 2', \dots, 12, 12'$. This time the degree of each of the six vertices $V_1,\dots,V_6$ is $4$, and the relations they induce are as in~\Lref{Lemma4pt}. For $V_1$ we have $(i,j,k,l)=(1,2,3,4)$, for $V_2$ we have $(i,j,k,l)=(3,6,8,11)$, for $V_3$ we have $(i,j,k,l)=(4,5,6,7)$, for $V_4$ we have $(i,j,k,l)=(7,10,11,12)$, for~$V_5$ the indices are $(i,j,k,l) = (2,5,9,10)$, and finally, for $V_6$, we have $(i,j,k,l) = (1,8,9,12)$.
The relations of type (5*) are:
$$[\G_{i,i'}, \G_{j,j'}]=e$$ for
$$(i,j) \in \begin{array}{l}
\Set{1}{5,6,7,10,11}\cup \Set{2}{6,7,8,11,12}\cup \Set{3}{5,7,9,10,12} \\
\cup \Set{4}{8,9,10,11,12}
\cup \Set{5}{8,11,12}\cup \Set{6}{9,10,12} \\
\cup\Set{7}{8,9}\cup \Set{8}{10}\cup\Set{9}{11}.
\end{array}$$
The projective relation is, as before, $\prodl_{i=12}^1  \G_{i'}\G_i = e$.

Group $G/\nsg{\gamma^2}$ now has 24 generators and a long list of relations. We used  \Magma\ to simplify the relations. Again,  some generators can be expressed in terms of others, and we are left with  the generators $\G_{2'}, \G_{3'}, \G_4, \G_5, \G_{5'}, \G_{6'}, \G_{7}, \G_{8'}, \G_{10},\G_{10'}$.
We label the triangles as in \Fref{X4++}, and  consider the dual graph which is generated by these generators, see \Fref{dual_graph_4_First_run}.
\begin{figure}[H]
$$\xymatrix@R=25pt@C=25pt{
{}  & V_5 \ar@{-}[dl]|9 \ar@{-}[d]|2 \ar@{-}[dr]|5 & {} & {} & V_5 \ar@{-}[dl]|5 \ar@{-}[d]|{10} \ar@{-}[dr]|9 & {} \\
V_6 \ar@{-}[r]|1 \ar@{-}[dr]|8 & V_1
\ar@{}[ru]|(0.3){\textcircled{c}}
\ar@{}[lu]|(0.3){\textcircled{b}}
\ar@{}[ld]|(0.3){\textcircled{g}}
\ar@{}[rd]|(0.3){\textcircled{h}}
\ar@{-}[r]|4 \ar@{-}[d]|3 & V_3 \ar@{-}[dl]|6 & V_3 \ar@{-}[r]|7 \ar@{-}[dr]|6 & V_4
\ar@{}[ru]|(0.3){\textcircled{a}}
\ar@{}[lu]|(0.3){\textcircled{d}}
\ar@{}[ld]|(0.3){\textcircled{e}}
\ar@{}[rd]|(0.3){\textcircled{f}}
\ar@{-}[r]|{12} \ar@{-}[d]|{11} & V_6 \ar@{-}[dl]|8 \\
{}  & V_2 & {} & {} & V_2 & {}
}$$
\caption{$X^4$ with labelled triangles}\label{X4++}
\end{figure}

\begin{figure}
$$
\xymatrix@R=30pt@C=30pt{
{} & {} & b \ar@{-}[r]^{2'} & c \ar@{-}[d]_{4} \ar@/^1ex/@{-}[r]<1.5pt>^{5} \ar@/_1ex/@{-}[r]<-1.5pt>_{5'} & d \ar@{-}[d]^7 \ar@/^1ex/@{-}[r]<1.5pt>^{10} \ar@/_1ex/@{-}[r]<-1.5pt>_{10'} & a  & \\
{} & f \ar@{-}[r]^{8'} & g \ar@{-}[r]^{3'}  & h \ar@{-}[r]^{6'} &
 e
}
$$
\caption{The dual graph after the first run in \Magma}\label{dual_graph_4_First_run}
\end{figure}
\forget
\begin{figure}[H]
	\begin{center}
		
		\definecolor{circ_col}{rgb}{0,0,0}
		\begin{tikzpicture}[x=1cm,y=1cm,scale=2]
		
		\draw [fill=circ_col] (0,0) node [below] {} circle (1pt);
		\draw [fill=circ_col] (1,0) node [below] {} circle (1pt);
		\draw [fill=circ_col] (0,1) node [below] {} circle (1pt);
		\draw [fill=circ_col] (1,1) node [below] {} circle (1pt);
		\draw [fill=circ_col] (2,1) node [below] {} circle (1pt);
		\draw [fill=circ_col] (2,0) node [below] {} circle (1pt);
		\draw [fill=circ_col] (-1,0) node [below] {} circle (1pt);
		\draw [fill=circ_col] (3,1) node [below] {} circle (1pt);
		
		\draw [-, line width = 1pt] (1,0) -- (1,1);
		\draw [-, line width = 1pt] (0,0) -- (1,0);
		\draw [-, line width = 1pt] (0,1) -- (1,1);
		\draw [-, line width = 1pt] (1,1) -- (2,1);
		\draw [-, line width = 1pt] (1,0) -- (2,0);
		\draw [-, line width = 1pt] (2,1) -- (2,0);
		\draw [-, line width = 1pt] (-1,0) -- (0,0);
		\draw [-, line width = 1pt] (2,1) -- (3,1);
		\draw [-, line width = 1pt] (1.1,0.9) -- (1.9,0.9);
		\draw [-, line width = 1pt] (2.1,0.9) -- (2.9,0.9);

                 \draw [-, line width = 1pt] (1,0) -- node [anchor=east]{$ 4 $} (1,1);
                 \draw [-, line width = 1pt] (0,0) -- node [below]{$ 3' $} (1,0);
                 \draw [-, line width = 1pt] (0,1) -- node [above]{$ 2' $} (1,1);
                 \draw [-, line width = 1pt] (1,1) -- node [above]{$ 5 $} (2,1);
\draw [-, line width = 1pt] (1,0) -- node [below]{$ 6' $} (2,0);
\draw [-, line width = 1pt] (2,0) -- node [anchor=west]{$ 7 $} (2,1);
                 \draw [-, line width = 1pt] (-1,0) -- node [below]{$ 8' $} (0,0);
                 \draw [-, line width = 1pt] (2,1) -- node [above]{$ 10 $} (3,1);
                 \draw [-, line width = 1pt] (1.1,0.9) -- node [below]{$ 5' $} (1.9,0.9);
                 \draw [-, line width = 1pt] (2.1,0.9) -- node [below]{$ 10' $} (2.9,0.9);

		\end{tikzpicture}
		
	\end{center}
	\setlength{\abovecaptionskip}{-0.15cm}
	\caption{The dual graph after the first run in \Magma}\label{dual_graph_4_First_run}
\end{figure}
\forgotten
We collected all relations that are of form (\ref{com}), noticing that relations $\langle \G_5, \G_{10'}\rangle=e, \langle \G_5, \G_{10}\rangle=e$, and  $\langle \G_{5}, \G_{7}\rangle=e$ are missing. Relation $\langle\G_{10'},\G_{5'}\G_{5}\G_{2'}\G_{5}\G_{5'}\G_{5}\G_{2'}\G_{5}\G_{5'} \rangle = e$ from the output can be simplified  to $\langle \G_5, \G_{10'}\rangle=e$, and this is enough for us to run \Magma\ on the original presentation, with this additional relation, which is now proven. We now get another simplified presentation for $G/\nsg{\gamma^2}$, this time with the generators $\G_{1'}, \G_{2'}, \G_{3'}, \G_5, \G_{5'},  \G_{7}, \G_{8'}, \G_9, \G_{11},\G_{11'}$, whose graph, denoted by $T$, is depicted in \Fref{dual_graph_4_second_run}.

\begin{figure}[H]
\[
\xymatrix@R=30pt@C=30pt{
{} a \ar@{-}[r]^{9}  &b \ar@{-}[d]^{1'} \ar@{-}[r]^{2'}  &c \ar@/^1ex/@{-}[r]<2pt>^5 \ar@/_1ex/@{-}[r]<-2pt>_{5'} &d \ar@{-}[d]^7  & \\
{} h \ar@{-}[r]^{3'}  &g \ar@{-}[r]^{8'} &f \ar@/_1ex/@{-}[r]<-2pt>_{11} \ar@/^1ex/@{-}[r]<2pt>^{11'}& e
}
\]
\caption{The dual graph after the second run in \Magma}\label{dual_graph_4_second_run}
\end{figure}
\forget
\begin{figure}[H]
	\begin{center}
		
		\definecolor{circ_col}{rgb}{0,0,0}
		\begin{tikzpicture}[x=1cm,y=1cm,scale=2]
		
		\draw [fill=circ_col] (0,0) node [below] {h} circle (1pt);
		\draw [fill=circ_col] (0,1) node [above] {a} circle (1pt);
		\draw [fill=circ_col] (1,0) node [below] {g} circle (1pt);
		\draw [fill=circ_col] (1,1) node [above] {b} circle (1pt);
		\draw [fill=circ_col] (2,0) node [below] {f} circle (1pt);
		\draw [fill=circ_col] (2,1) node [above] {c} circle (1pt);
		\draw [fill=circ_col] (3,0) node [below] {e} circle (1pt);
		%\draw [fill=circ_col] (4,0) node [below] {} circle (1pt);
		\draw [fill=circ_col] (3,1) node [above] {d} circle (1pt);
		%\draw [fill=circ_col] (4,1) node [below] {} circle (1pt);
		
		\draw [-, line width = 1pt] (0,0) -- (3,0);
		\draw [-, line width = 1pt] (0,1) -- (1,1);
		\draw [-, line width = 1pt] (1,1) -- (2,1);
		\draw [-, line width = 1pt] (2,1) -- (3,1);
		\draw [-, line width = 1pt] (2.1,0.9) -- (2.9,0.9);
		%\draw [-, line width = 1pt] (3,1) -- (4,1);
		\draw [-, line width = 1pt] (2,0) -- (3,0);
		%\draw [-, line width = 1pt] (3,0) -- (4,0);

                 \draw [-, line width = 1pt] (1,0) -- node [anchor=east]{$ 1' $} (1,1);
                 \draw [-, line width = 1pt] (0,0) -- node [below]{$ 8' $} (3,0);
                 \draw [-, line width = 1pt] (0,0) -- node [below]{$ 3' $} (1,0);
                 \draw [-, line width = 1pt] (0,1) -- node [above]{$ 9 $} (1,1);
                 \draw [-, line width = 1pt] (1,1) -- node [above]{$ 2' $} (2,1);
                 \draw [-, line width = 1pt] (2,1) -- node [above]{$ 5 $} (3,1);
                 \draw [-, line width = 1pt] (3,1) -- node [anchor=west]{$ 7 $} (3,0);
                 %\draw [-, line width = 1pt] (3,1) -- node [above]{$ 9 $} (4,1);
                 \draw [-, line width = 1pt] (2,0) -- node [below]{$ 11 $} (3,0);
                 \draw [-, line width = 1pt] (2.1,0.9) -- node [below]{$ 5' $} (2.9,0.9);
                 \draw [-, line width = 1pt] (2.1,0.1) -- node [above]{$ 11' $} (2.9,0.1);

		\end{tikzpicture}
		
	\end{center}
	\setlength{\abovecaptionskip}{-0.15cm}
	\caption{The dual graph after the second run in \Magma}\label{dual_graph_4_second_run}
\end{figure}
\forgotten{}

We want to verify that our group is a quotient of $C_Y(T)$, where $T$ is given in \Fref{dual_graph_4_second_run}.  First we notice that the relations $\langle \G_{2'}, \G_{9}\rangle=e,\langle \G_5, \G_{7}\rangle=e$, $[\G_5,\G_9] = e$ and $[\G_{5'},\G_9] = e$ are missing; but these can be deduced using other relations from the output; for example, we simplified $\langle \G_{2'} \G_{5'} \G_{2'}, \G_{9}\rangle=e$ to get $\langle \G_{2'}, \G_{9}\rangle=e$. Thereafter, we proved the following relations:
\begin{eqnarray*}
{}[\G_{1'}\G_{2'} \G_{1'},\G_{9}] &=& e \\
{}[\G_{1'}\G_{3'} \G_{1'},\G_{8'}] &=& e \\
{}\langle\G_{5'},\G_{2'}\G_{5}\G_{2'} \rangle &=& e \\
{}\langle\G_{5'},\G_{7}\G_{5}\G_{7} \rangle &=& e \\
{}\langle\G_{11'},\G_{7}\G_{11}\G_{7} \rangle &=& e \\
{}\langle\G_{11'},\G_{8'}\G_{11}\G_{8'} \rangle &=& e \\
{}[\G_{2'}\G_{5'} \G_{2'},\G_{7}\G_{5}\G_{7}] &=& e \\
{}[\G_{8'}\G_{11'} \G_{8'},\G_{7}\G_{11}\G_{7}] &=& e.
\end{eqnarray*}
By that, we verified all the defining relations \eq{CYT1}--\eq{CYT4}, so our group is indeed a quotient of $C_Y(T)$.

Let $F_{t,n}$ be the group defined in Subsection \ref{Vishne}.
As explained in Subsection~\ref{ss:CYT}, $C_Y(T) \,\cong\, S_8 \ltimes F_{3,8}$.
To compute $G/\nsg{\gamma^2}$, it remains to cast all the relations in the language of the generators of $F_{3,8}$. Recall that $F_{3,8}$ is a subgroup of the direct product $(\F_3)^8$; let $u_i, v_i, w_i$ be the generators of the $i$th component. So the free groups $\sg{u_i,v_i,w_i}$  commute elementwise.
In this language, $F_{3,8}$ is generated by all the elements $u_ju_i^{-1}$, $v_jv_i^{-1}$ and $w_jw_i^{-1}$.

Based on the spanning subgraph $T_0$ that was obtained from $T$ by removing edges $5$, $5'$ and $11'$, the final stage is to reinterpret the defining relations of $\bar{G} = G/\nsg{\gamma^2}$, which we do by mapping $\bar{G}$ to a quotient of $S_8 \ltimes F_{3,8}$ by
$$\G_{1'} \mapsto (b g),
\ \
\G_{2'} \mapsto (b c),
\ \
\G_{3'} \mapsto (g h),
\ \
\G_{7} \mapsto (d e),
\ \
\G_{8'} \mapsto (f g),
\ \
\G_{9} \mapsto (a b)$$
and
$$
\G_{5} \mapsto (c d)u_c u_d^{-1},
\ \
\G_{5'} \mapsto (c d) v_c v_d^{-1},
\ \
\G_{11'} \mapsto (e f) w_e w_f^{-1}.$$

As may be expected, most of the defining relations of $\bar{G}$ now become trivial.
One exception is the relation $\omega = e$, where
\begin{equation}\label{mechoar}
\omega = [u_d,v_d][v_c,w_c]u_a^{-1}w_av_b^{-1}u_bw_b^{-1}v_bu_cw_cw_d^{-1}u_d^{-1}w_e^{-1}u_e^{-1}w_fu_f^{-1}u_gw_gu_hw_h^{-1}.
\end{equation}
The other nontrivial relations are much nicer, and by acting with $S_8$, we obtain the following rules:
\begin{eqnarray}
{}[u_i,v_i]&=&[u_j,v_j], \label{D1}\\
{} [u_i,w_i]&=&[u_j,w_j],  \\
{}[v_i,w_i]&=&[v_j,w_j], \label{D3}\\
{} (u_i)^2&=&(u_j)^2, \label{D4}\\
{} (v_i)^2&=&(v_j)^2, \\
{} (w_i)^2&=&(w_j)^2.  \label{D6}
\end{eqnarray}

Let $C^*$ and $C$, respectively, denote the quotient of $F^*_{3,8}$ and $F_{3,8}$ by the relations \eq{D1}--\eq{D6}. It is easier to compute in $C^*$: Write $[u,v]$ for all $[u_i,v_i]$, and likewise for $[u,w]$, $[v,w]$, and the squares $u^2, v^2, w^2$. We have that $v_i^2
= v_j^2 = u_iv_j^2u_i^{-1} = u_iv_i^2u_i^{-1}=[u,v]^2v_i^2$, so $[u,v]^2 = 1$, and similarly, $[u,w]^2 = [v,w]^2 = 1$.  Modulo the commutators $[u,v]$, $[v,w]$, and $[w,u]$, which are central, the group is abelian.
\begin{cor}\label{C,C*}
Groups $C$ and $C^*$ are central extensions, as described in the following commutative diagram:
$$\xymatrix@C=12pt@R=12pt{
{} & 1 \ar@{->}[d] & 1  \ar@{->}[d]& 1 \ar@{->}[d] & {} \\
1 \ar@{->}[r] & \Z_2^3 \ar@{->}[r] \ar@{->}[d] & C \ar@{->}[r]^{\operatorname{ab}} \ar@{->}[d] & (\Z_2^3)^7  \ar@{->}[r] \ar@{->}[d] & 1 \\
1 \ar@{->}[r] & \Z_2^3 \times \Z^3 \ar@{->}[r] \ar@{->}[d] & C^* \ar@{->}[r]^{\operatorname{ab}} \ar@{->}[d] & (\Z_2^3)^8 \ar@{->}[r] \ar@{->}[d] & 1  \\
1 \ar@{->}[r] & \Z^3 \ar@{->}[r]^{\cdot 2} \ar@{->}[d] & \Z^3 \ar@{->}[r] \ar@{->}[d] & \Z_2^3 \ar@{->}[r] \ar@{->}[d]& 1 \\
{} & 1 & 1 & 1 & {}
}$$
Here the group at the upper-left corner is generated by the commutator;
the arrows so-marked are the abelianization maps, and the middle row is induced by the short exact series
$$1 \rightarrow F_{3,8} \rightarrow F_{3,8}^* \rightarrow \mathbb{Z}^3 \rightarrow 1$$ described in \Dref{Ftndef}.
\end{cor}

Notice that (for distinct $i,j,k$) $u_iu_j^{-1}$ commutes with $v_iv_j^{-1}$ in $C$, but $[u_iu_j^{-1},v_iv_k^{-1}] = [u,v]$. In light of \Cref{C,C*} we can simplify the element $\omega$ defined in \eq{mechoar}. The $b$-entry of this element relation cancels the commutators, and we get
$\omega = u_a^{-1}w_a
u_bw_b^{-1}
u_cw_c
w_d^{-1}u_d^{-1}
w_e^{-1}u_e^{-1}
w_fu_f^{-1}
u_gw_g
u_hw_h^{-1}$.  Using $u_i^2 = u_j^2$ and $w_i^2 = w_j^2$, as well as $[u,w]^2 = 1$. We then get a further simplification to
\begin{equation}\label{mechoarpachot}
\omega = [u,w]
(u_au_bu_cu_d
u_e^{-1}u_f^{-1}u_g^{-1}u_h^{-1})
(w_aw_bw_cw_d
w_e^{-1}w_f^{-1}w_g^{-1}w_h^{-1}).
\end{equation}
Again, by \eq{D4} and \eq{D6}, this element is symmetric under the action of $S_8$ on the indices, and has order $2$, so $C/\sg{\omega}$ has order $2^{23}$.
Recall that this is the group $\pi_1(\Galois{X^4})$, whose computation is the main goal of this paper.

Recalling that the fundamental group of the complement of the branch curve is the semidirect product of this group with $S_8$, it is desirable to describe the $S_8$-module structure of the abelianization.

\begin{rmk}
Considering all groups as $S_8$-modules, the abelianization $C^*/[C^*,C^*]$ is a direct sum of three copies of the natural module $V = \Z_2^8$, namely $\sg{u_i}$, $\sg{v_i}$, and $\sg{w_i}$.

For the sake of comparison, note that for $p>2$, $\Z_p^8$ is a direct sum of two irreducible submodules, $\Z_p^8 = (\Z_p^8)_0 \oplus D$, where the left component is the subspace of zero-sum vectors and $D$ is the ``diagonal'' subspace, spanned by $(1,\dots,1)$. On the other hand, for $p = 2$ the irreducible modules are  $0 \subset D \subset V_0 \subset V$, of dimensions $0<1<7<8$ over $\Z_2$, respectively. Now, while $C^*/[C^*,C^*] = V \oplus V \oplus V$, it follows from \Cref{C,C*} that the abelianization of $C$ is $C/[C,C] \cong V_0 \oplus V_0 \oplus V_0$.

Finally, $\omega$ glues the $u$ and $w$ copies of $D$ in this sum, so that the abelianization of $\pi_1(\Galois{X^4})$, is  $$(C/\sg{\omega})/[C/\sg{\omega},C/\sg{\omega}] \cong V_0 \oplus (V_0 \oplus V_0)/D(1,1),$$
where $D(1,1)$ is the diagonal copy of $D$ in the direct sum $D \oplus D \subset V_0 \oplus V_0$.
\end{rmk}

%%%%%%%%%%%%%%%%%%%%%  FORGET
\forget
%From here a mistake, we can even delete this part.
We now define the map from the group $G/\nsg{\gamma^2}$ to $S_8 \ltimes A_{2,8}$ by setting
$$\G_{1'} \mapsto (b\ g), \G_{2'} \mapsto (b\ c), \G_{3'} \mapsto (g\ h), \G_{5} \mapsto (c\ d)u_{c,d}, \G_{5'} \mapsto (c\ d)u_{c,d}', \G_7 \mapsto (d\ e), \G_{8'} \mapsto (f\ g), \G_9 \mapsto (a\ b),$$
and $$\G_{11} \mapsto (e\ f),\G_{11'} \mapsto \G_{11'}.$$
Now, we show that $u_{x,y} = e$ for all $x,y$ and $u_{w,z}' = e$ for all $w,z$. To prove this we list down all the required relations in $S_8 \ltimes A_{2,8}$ among the second \Magma\ output.
\begin{itemize}
    \item[1.] $\G_{11'}\G_{11'} = e$, \ \ 2. $[(b\ g), \G_{11'}] = e$, \ \ 3. $[(a\ b), \G_{11'}] = e$, \ \ 4. $[(b\ c), \G_{11'}] = e$, \ \ 5. $[(g\ h), \G_{11'}] = e$
    \item[6.]  $[(d\ g), \G_{11'}] = e \quad ({\overset{2}\Rightarrow}\ [(d\ b), \G_{11'}] = e\  {\overset{3}\Rightarrow}\ [(d\ a), \G_{11'}] = e)$
    \item[7.] $[(c\ d)u_{c,d}, \G_{11'}] = e \quad ({\overset{4}\Rightarrow}\ [(b\ d)u_{b,d}, \G_{11'}] = e \ {\overset{6}\Rightarrow}\ [u_{b,d}, \G_{11'}] = e \  {\overset{3}\Rightarrow}\ [u_{a,d}, \G_{11'}] = e )$
    \item[8.] $[u_{d,c}'u_{c,d},\G_{11'}] = e \quad (\underset{(u_{b,d} =e)}{\overset{4}\Rightarrow}\ [u_{d,b}', \G_{11'}] = e)$.
\end{itemize}
We note the following relation from the \Magma:
\begin{eqnarray*}
(a\ d\ e)u_{a,c}'u_{c,e}u_{e,a}'u_{a,c}u_{c,e}'\G_{11'}(a\ b)(d\ e) \G_{11'}(a\ e\ d\ b) u_{a,c}'u_{c,d}u_{d,a}'u_{a,c}u_{c,d}'\G_{11'} = e.
\end{eqnarray*}
By applying the isomorphism of $A_{2,8} \cong F_{2,8}$ (see \cite[Theorem 5.7]{RTV})  and using $2,4,6$ and $7$ we simplify the above relation:
\begin{eqnarray*}
(a\ d\ e) u_{c}'^{-1}u_{a}' u_{e}^{-1}u_{c} u_{a}'^{-1}u_{e}' u_{c}^{-1}u_{a} u_{e}'^{-1}u_{c}'
\G_{11'}(a\ b)(d\ e) \G_{11'}(a\ e\ d\ b)
u_{c}'^{-1}u_{a}' u_{d}^{-1}u_{c} u_{a}'^{-1}u_{d}'
u_{c}^{-1}u_{a} \quad \quad \cr
u_{d}'^{-1}u_{c}' \G_{11'} = e \cr
\Rightarrow (a\ d\ e) u_{c}'^{-1} u_{e}^{-1}u_{c} u_{e}' u_{c}^{-1}u_{a} u_{e}'^{-1}u_{c}'
\G_{11'}(a\ b)(d\ e) \G_{11'}(a\ e\ d\ b)
u_{c}'^{-1}u_{a}' u_{d}^{-1}u_{c} u_{a}'^{-1} u_{c}^{-1}u_{a} u_{c}' \G_{11'} = e \cr
\Rightarrow (a\ d\ e) u_{c}'^{-1} u_{e}^{-1}u_{c} u_{c}^{-1}u_{a} u_{c}'
\G_{11'}(a\ b)(d\ e) \G_{11'}(a\ e\ d\ b)
u_{c}'^{-1}u_{a}' u_{d}^{-1} u_{a}'^{-1} u_{a} u_{c}' \G_{11'} = e \cr
\Rightarrow (a\ d\ e) u_{c}'^{-1} u_{e}^{-1}u_{a} u_{c}'
\G_{11'}(a\ b)(d\ e) \G_{11'}(a\ e\ d\ b)
u_{c}'^{-1} u_{d}^{-1} u_{a} u_{c}' \G_{11'} = e \cr
\Rightarrow (a\ d\ e) u_{e}^{-1}u_{a} \G_{11'}(a\ b)(d\ e) \G_{11'}(a\ e\ d\ b) u_{d}^{-1} u_{a}  \G_{11'} = e \cr
\Rightarrow u_{d}^{-1}u_{e} (a\ d) (d\ e) \G_{11'}(d\ e) (a\ b) \G_{11'}(d\ e)(a\ d\ b) u_{d}^{-1} u_{a}  \G_{11'} = e \cr
\Rightarrow u_{d}^{-1}u_{e} (a\ d) (a\ b) (d\ e)\G_{11'}(d\ e)\G_{11'}(d\ e) (a\ d\ b) u_{d}^{-1} u_{a}  \G_{11'} = e \cr
\Rightarrow u_{d}^{-1}u_{e} (a\ d) (a\ b) \G_{11'} (a\ d\ b) u_{d}^{-1} u_{a}  \G_{11'} = e \cr
\Rightarrow u_{d}^{-1}u_{e} (a\ d) (a\ b) \G_{11'} (a\ b) (a\ d) u_{d}^{-1} u_{a}  \G_{11'} = e \cr
\Rightarrow u_{d}^{-1}u_{e} (a\ d) (a\ b) (a\ b) (a\ d) \G_{11'} u_{d}^{-1} u_{a}  \G_{11'} = e \cr
\Rightarrow u_{e,d} \G_{11'} u_{a,d}  \G_{11'} = e \cr
\Rightarrow u_{e,d} \G_{11'}  u_{a,d}  \G_{11'} = e \cr
\Rightarrow u_{e,d}  u_{a,d}   = e.
\end{eqnarray*}
Conjugating the above relation by $(a\ c)$ we obtain $u_{e,d} =  u_{d,c}$ and by substituting back we get
\[
u_{a,c} = u_{d,c} u_{a,d} = u_{e,d}  u_{a,d}  = e.
\]
Hence $u_{x,y} = e$ for all $x,y$. Now, we modify the group map from $G/\nsg{\gamma^2}$ to $S_8 \ltimes A_{2,8}$ by setting
\[
\G_{5} \mapsto (c\ d), \G_{11} \mapsto (e\ f)v_{e,f}, \G_{11'} \mapsto (e\ f)v_{e,f}',
\]
and translating the second \Magma\ output under this new map. We get one of the relation as $v_{b,g}^2 = e$, which in particular gives $v_{x,y}^2 = e$ for all $x,y$. Applying the isomorphism $A_{2,8} \cong F_{2,8}$ defined by $v_{x,y}' \mapsto v_y'^{-1}v_x', v_{w,z} \mapsto v_z^{-1}v_w$ \cite[Theorem 5.7]{RTV}, we notice that $v_z^2 = e$ holds in $F_{2,8}$ as well. We show that $v_{x,y}'^2 = e$ as well. To this end we pick the following relation and solve to get
\begin{eqnarray*}
 v_{f,d}'v_{d,a}v_{a,e}'v_{d,e}v_{h,b}v_{h,a}v_{c,h}'v_{e,c}v_{a,b}'v_{b,g}v_{g,a}'v_{c,g} = e\cr
\Rightarrow v_{f,d}'v_{d,a}v_{a,e}'v_{d,e}v_{a,b}v_{h,a}^2v_{c,h}'v_{e,c}v_{a,b}'v_{c,g}v_{b,c}v_{g,a}'v_{c,g} = e\cr
\Rightarrow v_{f,d}'v_{d,a}v_{a,e}'v_{d,e}v_{a,b}v_{c,h}'v_{e,c}v_{c,g}v_{a,b}'v_{g,a}'v_{b,c}v_{c,g} = e \cr
\Rightarrow v_{f,d}'v_{d,a}v_{a,e}'v_{d,e}v_{a,b}v_{c,h}'v_{e,g}v_{g,b}'v_{b,g} = e \cr
\Rightarrow v_{a,g}v_{f,d}'v_{d,a}v_{a,e}'v_{d,e}v_{a,b}v_{e,g}v_{c,h}'v_{g,b}'v_{b,a} = e \cr
\Rightarrow v_{f,d}'v_{a,g}v_{d,a}v_{a,e}'v_{d,e}v_{a,b}v_{e,g}v_{c,h}'v_{g,b}'v_{b,a} = e \cr
\Rightarrow v_{f,d}'v_{a,e}'v_{d,g}v_{d,e}v_{a,b}v_{e,g}v_{c,h}'v_{g,b}'v_{b,a} = e \cr
\Rightarrow v_{f,d}'v_{a,e}'v_{e,g}v_{d,e}^2v_{a,b}v_{e,g}v_{c,h}'v_{g,b}'v_{b,a} = e \cr
\Rightarrow v_{f,d}'v_{a,e}'v_{a,b}v_{c,h}'v_{g,b}'v_{b,a} = e.
\end{eqnarray*}
Now we apply the isomorphism $A_{2,8} \cong F_{2,8}$ to obtain
\begin{eqnarray*}
v_{d}'^{-1}v_{f}' v_{e}'^{-1}v_{a}' v_{b}^{-1}v_{a} v_{h}'^{-1}v_{c}' v_{b}'^{-1}v_{g}' v_{a}^{-1}v_{b} = e \cr
\Rightarrow v_{b}^{-1} v_{d}'^{-1}v_{f}' v_{e}'^{-1}v_{a}' v_{h}'^{-1}v_{c}' v_{b}'^{-1}v_{g}' v_{a} v_{a}^{-1}v_{b} = e \cr
\Rightarrow v_{d}'^{-1}v_{f}' v_{e}'^{-1}v_{a}' v_{h}'^{-1}v_{c}' v_{b}'^{-1}v_{g}'  = e. \cr
\end{eqnarray*}
Therefore, we get $v_{f,d}'v_{a,e}'v_{c,h}'v_{g,b}' = e$ in $A_{2,8}$. Conjugating by the transposition $(f,d)$ and subtracting from it, we get
\begin{eqnarray*}
v_{f,d}' - v_{d,f}' = e\cr
\Rightarrow v_{f,d}'^2 = e.
\end{eqnarray*}
Hence, by appropriately conjugating, we get $v_{x,y}'^2 = e$ for all $x,y$ in $A_{2,8}$ and $v_{x}'^2 = e$ for all $x$ in $F_{2,8}$. Now we focus on the following groups
\[
H := \frac{F_{2,8}}{<v_x^2=v_y^2,v_x'^2=v_y'^2 >}\quad \textit{and} \quad G:= \frac{\prodl_{1}^{8}{<v_x,v_x'>}}{<v_x^2=v_y^2,v_x'^2=v_y'^2 >}
\]
Let $N:= <v_x^2,v_x'^2| \ v_x,v_x' \in G>$. Then it is clear from the definition of $G$ that $N \subset C(G)$ and is a normal
subgroup of $G$. Also, it can be seen that $H \cap N = \{e\}$ and $H\cdot N = G$. From the second isomorphism theorem, we get
\[
H = \frac{H}{H\cap N} \cong \frac{G}{N}.
\]
\textbf{Claim:} $H := \mathbb{Z}_2^8 \ltimes \mathbb{Z}_2^7 $.
Notice that $G$ is a direct product of $8$ copies of infinite dihedral group $D_{\infty}$, where $D_{\infty}$ is defined as a semidirect
product $\mathbb{Z} \ltimes \mathbb{Z}_2$ with the action $\overline{1}: n \mapsto -n$, for $\overline{1} \in \mathbb{Z}_2$ and $n \in \mathbb{Z}$. Therefore
\[
G \cong D_{\infty} \times \dots \times D_{\infty} \cong \mathbb{Z}^8 \ltimes \mathbb{Z}_2^8,
\]
where the action of $\mathbb{Z}_2^8$ on $\mathbb{Z}^8$ is given by
\[
(\overline{a_i})_i \cdot (x_j)_j = (\overline{a_i} \cdot x_i)_i
\]
for $(\overline{a_i})_i \in \mathbb{Z}_2^8$ and $(x_j)_j \in \mathbb{Z}^8$. Now we look at the image of $H$ in $G$. For that let $\sigma_2$ denote the hyperplane in $\mathbb{Z}_2^8$ defined by $\sum \overline{a_i} = 0$ and $\sigma_{\mathbb{Z}}$ to be a
hyperplane in $\mathbb{Z}^8$ defined by $\sum x_i = \textit{even}$. Then $\sigma_2$ acts on $\sigma_{\mathbb{Z}}$ by the same
rule as shown above, and we get
\[
\sigma_{\mathbb{Z}} \ltimes \sigma_2 = \frac{A_{2,8}}{<v_{xy}^2,v_{xy}'^2,>} = H.
\]
Notice that, $v_i\cdot v_i' = x_i$ and $v_i = a_i$. In the new generators, the commutator relation becomes
$v_iv_i'v_i^{-1}v_i'^{-1} = v_iv_i'v_iv_i' = 2x_i$. Therefore we get $2x_i = 2x_j$ which corresponds to vectors of the form
$(2,-2,2,-2,\dots,2,-2) \in \mathbb{Z}_8$. Conjugating $(2,-2,2,-2,\dots,2,-2)$ by vector $(0,1,1,\dots,0) \in \mathbb{Z}_2^8$ and its inverse
we obtain $(2,2,0,\dots,0)$. Therefore we mod-out these vectors from $\sigma_{\mathbb{Z}} \ltimes \sigma_2$. In order to do so, we first define a subgroup (normal) of $\sigma_{\mathbb{Z}}$
\[
P := \{(x_1,x_2,\dots.,x_8) \in \sigma_{\mathbb{Z}} | \ 4 \ \text{devides} \ \sum x_i \ \text{and} \ 2\ \text{devides} \ x_j \ \text{for all $j$} \}
\]
We claim now that $\sigma_{\mathbb{Z}}/P \cong \mathbb{Z}_2^8$. To see this we choose a basis of $\sigma_{\mathbb{Z}}$ to be $\{v_1:= x_1-x_2, v_2:= x_1-x_3,\dots, v_7:= x_1-x_8, v_8:= 2x_1\}$ and that of $\sigma_2$ to be $\{w_1:= a_1+a_2, w_2:= a_1+a_3,\dots, w_7:= a_1+a_8 \}$. Under this setting, $P$ is generated by $\{2v_1,2v_2,\dots,2v_8 \}$ and therefore we get
$\sigma_{\mathbb{Z}}/P \cong \mathbb{Z}_2^8$. Action of $\sigma_2$ on
$\sigma_{\mathbb{Z}}$ is given by
\[
w_i \cdot v_j =
         \begin{cases}
             -v_i \ \text{if $j=i$} \cr
             v_8 - v_j \ \text{if $j (\ne i) < 8$} \cr
             -v_8 \ \text{if $j=8$}.
        \end{cases}
\]
After moding out by $P$, the above action becomes
\[
w_i \cdot v_j =
         \begin{cases}
             v_i \ \text{if $j=i$} \cr
             v_8 + v_j \ \text{if $j (\ne i) < 8$} \cr
             v_8 \ \text{if $j=8$}.
        \end{cases}
\]
and final group becomes $\mathbb{Z}_2^8 \ltimes \mathbb{Z}_2^7$.

\forgotten

\end{proof}

%+-------------------------------------------+
%|						                     |
%|    		BIBLIOGRAPHY		             |
%|						                     |
%+-------------------------------------------+
%\newpage
%\cleardoublepage
%\phantomsection

\end{document}